\documentclass[leqno]{article}
\usepackage{geometry}
\usepackage{graphicx}	
\usepackage[cp1251]{inputenc}
\usepackage[english]{babel}
\usepackage{mathtools}
\usepackage{amsfonts,amssymb,mathrsfs,amscd,amsmath,amsthm}

\usepackage{verbatim}

\usepackage{url}

\usepackage{stmaryrd}

\def\ig#1#2#3#4{\begin{figure}[!ht]\begin{center}%
\includegraphics[height=#2\textheight]{#1.eps}\caption{#4}\label{#3}%
\end{center}\end{figure}}

\def\thtext#1{
  \catcode`@=11
  \gdef\@thmcountersep{. #1}
  \catcode`@=12
}

\def\threst{
  \catcode`@=11
  \gdef\@thmcountersep{.}
  \catcode`@=12
}

\theoremstyle{plain}
\newtheorem*{mainthm}{Main Theorem}
\newtheorem{thm}{Theorem}[section]
\newtheorem{prop}[thm]{Proposition}
\newtheorem{cor}[thm]{Corollary}
\newtheorem{ass}[thm]{Assertion}
\newtheorem{lem}[thm]{Lemma}

\theoremstyle{definition}
\newtheorem{dfn}[thm]{Definition}
\newtheorem{rk}[thm]{Remark}
\newtheorem{constr}[thm]{Construction}

 \catcode`@=11
 \def\.{.\spacefactor\@m}
 \catcode`@=12

\def\N{{\mathbb N}}

\def\R{\mathbb R}

\def\a{\alpha}
\def\b{\beta}
\def\e{\varepsilon}
\def\dl{\delta}
\def\D{\Delta}
\def\g{\gamma}
\def\G{\Gamma}

\def\l{\lambda}
\def\r{\rho}
\def\s{\sigma}

\def\v{\varphi}

\def\0{\emptyset}
\def\:{\colon}
\def\<{\langle}
\def\>{\rangle}
\def\[{\llbracket}
\def\]{\rrbracket}

\def\c{\circ}

\def\rom#1{\emph{#1}}
\def\({\rom(}
\def\){\rom)}
\def\sm{\setminus}
\def\ss{\subset}

\def\sp{\supset}

\def\x{\times}

\def\diam{\operatorname{diam}}

\def\dis{\operatorname{dis}}

\def\Gr{\operatorname{Gr}}

\def\Mid{\operatorname{Mid}}

\def\opt{{\operatorname{opt}}}

\def\cB{{\cal B}}
\def\cC{{\cal C}}
\def\cD{{\cal D}}
\def\cF{{\cal F}}

\def\cM{{\cal M}}

\def\cP{{\cal P}}
\def\cR{{\cal R}}
\def\cS{{\cal S}}

\begin{document}
\title{Isometry Group of Gromov--Hausdorff Space}
\author{A.~O.~Ivanov, A.~A.~Tuzhilin}
\date{}
\maketitle

\begin{abstract}
The present paper is devoted to investigation of the isometry group of the Gromov--Hausdorff space, i.e., the metric space of compact metric spaces considered up to isometry and endowed with the Gromov--Hausdorff metric. The main goal is to present a proof of the following theorem by George Lowther (2015): \emph{the isometry group of the Gromov--Hausdorff space is trivial}~\cite{blog}. Unfortunately, the author himself has not publish an accurate text for 2 years passed from the publication of the draft~\cite{blog} that is full of excellent ideas mixed with unproved and wrong statements.
\end{abstract}

\section*{Introduction}
\markright{\thesection.~Introduction}
The present paper deals with the isometry group of the Gromov--Hausdorff space, i.e., the metric space consisting of isometry classes of compact metric spaces and endowed with the Gromov--Hausdorff metric (exact definitions can be found below, or, for instance, in ~\cite{BurBurIva}).

In 2015 during a discussion with Stavros Iliadis we paid attention to a conjecture that Stavros, according to him, proposed a few years ago: \emph{the isometry group of the Gromov--Hausdorff space is trivial}. To start with we have googled a blog~\cite{blog}, where Noah Schweber formulated the same conjecture\footnote{Here is the reference from the blog: ``Noah Schweber --- a postdoc at UW-Madison (previously a grad student at UC-Berkeley), interested in mathematical logic --- specifically, computability theory and reverse mathematics, set theory, and abstract model theory. I'm also interested in other Nifty Things, in mathematics and elsewhere''.}. Among comments to the conjecture we have found a positive ``solution'' given by George Lowther\footnote{The reference from the same blog: ``Apparently, this user prefers to keep an air of mystery about them''.}. We started to read the proof immediately and very soon met a number of obstacles: it turns out that many proofs are missed (that is natural for a draft), but there are some wrong statements also. The most striking example: it is stated that the Gromov--Hausdorff distance between finite metric spaces of the same cardinality is attained on a bijection (but simple computer simulation shows that it is not true in general case).

The present paper is a result of our critical retreat of the ideas from~\cite{blog}. However,  many constructions and proofs present here are not contained in~\cite{blog}. Some statements from~\cite{blog} we reformulate in a correct way. Besides that, we give geometrical interpretations to some constructions from~\cite{blog} that looked rather formal. Unfortunately, the author of the proof from~\cite{blog} did not publish an accurate text for more than 2 years passed.

Thus, our main goal is to prove the following theorem by George Lowther.

\begin{mainthm}
The isometry group of the Gromov--Hausdorff space is trivial.
\end{mainthm}

Let us mention that in the local sense the Gromov--Hausdorff space is rather symmetric, in particular, sufficiently small balls of the same radius centered at $n$-point general position spaces (the ones, where all non-zero distances are pairwise different and all triangle inequalities are strict) are isometric to each other~\cite{IvaTuzLocalStrIsom}. There it is also shown that the isometry group of sufficiently small balls centered at such spaces contains a subgroup isomorphic to the permutation group of an $n$-element set.

The authors are thankful to Stavros Iliadis for attracting their attention to this beautiful problem, and for many fruitful discussions. Also, the authors are thankful to George Lowther for brilliant ideas presented in~\cite{blog}.

\section{Main Definitions and Preliminary Results}
\markright{\thesection.~Main Definitions and Preliminary Results}
Let $X$ be a set. By $\#X$ we denoted the \emph{cardinality\/} of $X$.

Now, let $X$ be an arbitrary metric space. The distance between its points $x$ and $y$ is denoted by $|xy|$. If $A,B\ss X$ are nonempty subsets, then put $|AB|=\inf\bigl\{|ab|:a\in A,\,b\in B\bigr\}$. If $A=\{a\}$, then we write  $|aB|=|Ba|$ instead of $|\{a\}B|=|B\{a\}|$.

Let us fix the notations for the following standard objects related to a metric space $X$:
\begin{itemize}
\item for $x\in X$ and $r>0$ by $U_r(x)=\{y\in X:|xy|<r\}$ we denote the \emph{open ball centered at $x$ of radius $r$};
\item for $x\in X$ and $r\ge 0$  by $B_r(x)=\{y\in X:|xy|\le r\}$ and $S_r(x)=\{y\in X:|xy|=r\}$ we denote the \emph{closed ball\/} and the \emph{sphere centered at $x$ of radius $r$}, respectively;
\item for nonempty $A\in X$ and $r>0$ by $U_r(A)=\{x\in X:|xA|<r\}$ we denote the \emph{open neighbourhood of $A$ of radius $r$};
\item for nonempty $A\in X$ and $r\ge 0$  by $B_r(A)=\{x\in X:|xA|\le r\}$ and $S_r(A)=\{x\in X:|xA|=r\}$ we denote the \emph{closed neighbourhood\/} and the \emph{equidistant set of $A$ of radius $r$}.
\end{itemize}

\subsection{Hausdorff and Gromov--Hausdorff Distances}
For nonempty $A,\,B\ss X$ we put
$$
d_H(A,B)=\inf\bigl\{r>0:A\ss U_r(B)\ \text{and}\ B\ss U_r(A)\bigr\}=\max\{\sup_{a\in A}|aB|,\sup_{b\in B}|Ab|\}.
$$
This value is called the \emph{Hausdorff distance between $A$ and $B$}. It is well-known~\cite{BurBurIva} that the Hausdorff distance, being restricted to the set of all nonempty closed bounded subsets of $X$, forms a metric.

Let $X$ and $Y$ be metric spaces. A triple $(X',Y',Z)$ consisting of a metric space $Z$ and two its subsets $X'$ and $Y'$ that are isometric to $X$ and $Y$, respectively, is called a \emph{realization of the pair $(X,Y)$}. The \emph{Gromov--Hausdorff distance $d_{GH}(X,Y)$ between $X$ and $Y$} is the infimum of $r$ such that there exists a realization $(X',Y',Z)$ of the pair $(X,Y)$ with $d_H(X',Y')\le r$.

By $\cM$ we denote the set of all compact metric spaces considered up to an isometry.

\begin{thm}[\cite{BurBurIva}, \cite{IvaNikolaevaTuz}]\label{thm:general_props}
Being restricted on $\cM$, the distance $d_{GH}$ is a metric. The metric space $\cM$ is complete, separable, and geodesic\/ \(given any two points, there is a path between them, whose length equals to the distance between the points\/\).
\end{thm}

The next result is an immediate consequence of the definitions.

\begin{prop}[\cite{BurBurIva}]\label{prop:e-net-GH}
For an arbitrary nonempty subset  $Y$ of a metric space $X$ the inequality $d_{GH}(X,Y)\le d_H(X,Y)$ holds. In particular, if $Y$ is an $\e$-net in $X$, then $d_{GH}(X,Y)\le\e$.
\end{prop}

To calculate the Gromov--Hausdorff distance it is convenient to use the technique of cor\-res\-pon\-den\-ces.

Let $X$ and $Y$ be any nonempty sets. Recall that a \emph{relation\/} between $X$ and $Y$ is a subset of the Cartesian product $X\x Y$. By $\cP(X,Y)$  we denote the set of all \textbf{nonempty\/} relations between $X$ and $Y$. Let us consider each relation $\s\in\cP(X,Y)$ as a multivalued mapping, whose domain could be less that the whole set $X$. Then, similarly to the case of mappings, for each $x\in X$ and any $A\ss X$ their images $\s(x)$ and $\s(A)$ are defined, and for each $y\in Y$ and any $B\ss Y$ their pre-images $\s^{-1}(y)$ and $\s^{-1}(B)$ are defined as well.

A relation $R\in\cP(X,Y)$ is called a \emph{correspondence} if, being restricted onto $R$, the canonical projections $\pi_X\:(x,y)\mapsto x$ and $\pi_Y\:(x,y)\mapsto y$ are surjective, or, that is equivalent, if $R(X)=Y$ and $R^{-1}(Y)=X$. By $\cR(X,Y)$  we denote the set of all correspondences between $X$ and $Y$.

Let $X$ and $Y$ be arbitrary metric spaces. The value
$$
\dis\s=\sup\Bigl\{\bigl||xx'|-|yy'|\bigr|: (x,y),(x',y')\in\s\Bigr\}.
$$
is called the \emph{distortion $\dis\s$ of a relation $\s\in\cP(X,Y)$}.

\begin{prop}[\cite{BurBurIva}]
For any metric spaces $X$ and $Y$ it holds
$$
d_{GH}(X,Y)=\frac12\inf\bigl\{\dis R:R\in\cR(X,Y)\bigr\}.
$$
\end{prop}

The technique  of correspondences can simplify proofs of various well-known facts.

For any metric space $X$ and a real number $\l>0$ by $\l X$ we denote the metric space obtained from $X$ by multiplication of all the distances by $\l$.

\begin{prop}[\cite{BurBurIva}]\label{prop:GH_simple}
Let $X$ and $Y$ be metric spaces. Then
\begin{enumerate}
\item\label{prop:GH_simple:1} if $X$ is the single-point metric space, then $d_{GH}(X,Y)=\frac12\diam Y$\rom;
\item\label{prop:GH_simple:2} if $\diam X<\infty$, then
$$
d_{GH}(X,Y)\ge\frac12|\diam X-\diam Y|;
$$
\item\label{prop:GH_simple:3} $d_{GH}(X,Y)\le\frac12\max\{\diam X,\diam Y\}$, in particular, $d_{GH}(X,Y)<\infty$ for bounded $X$ and $Y$\rom;
\item\label{prop:GH_simple:4} for any $X\in\cM$ and any $\l\ge0$, $\mu\ge0$ we have $d_{GH}(\l X,\mu X)=\frac12|\l-\mu|\diam X$\rom; this immediately implies that the curve $\g(t):=t\,X$ is a shortest one for any pair of its points\rom;
\item\label{prop:GH_simple:5} for any $X,Y\in\cM$ and any $\l>0$ we have $d_{GH}(\l X,\l Y)=\l d_{GH}(X,Y)$. Moreover, for $\l\ne1$ the unique space that remains the same under this operation is the single-point space. In other words, the multiplication of a metric by a number $\l>0$ is a homothety of the space $\cM$ with the center at the single-point metric space.
\end{enumerate}
\end{prop}

Thus, the Gromov--Hausdorff space looks like a cone with the vertex at the single-point space, and with generators that are geodesics, see Figure~\ref{fig:gh-space}.

\ig{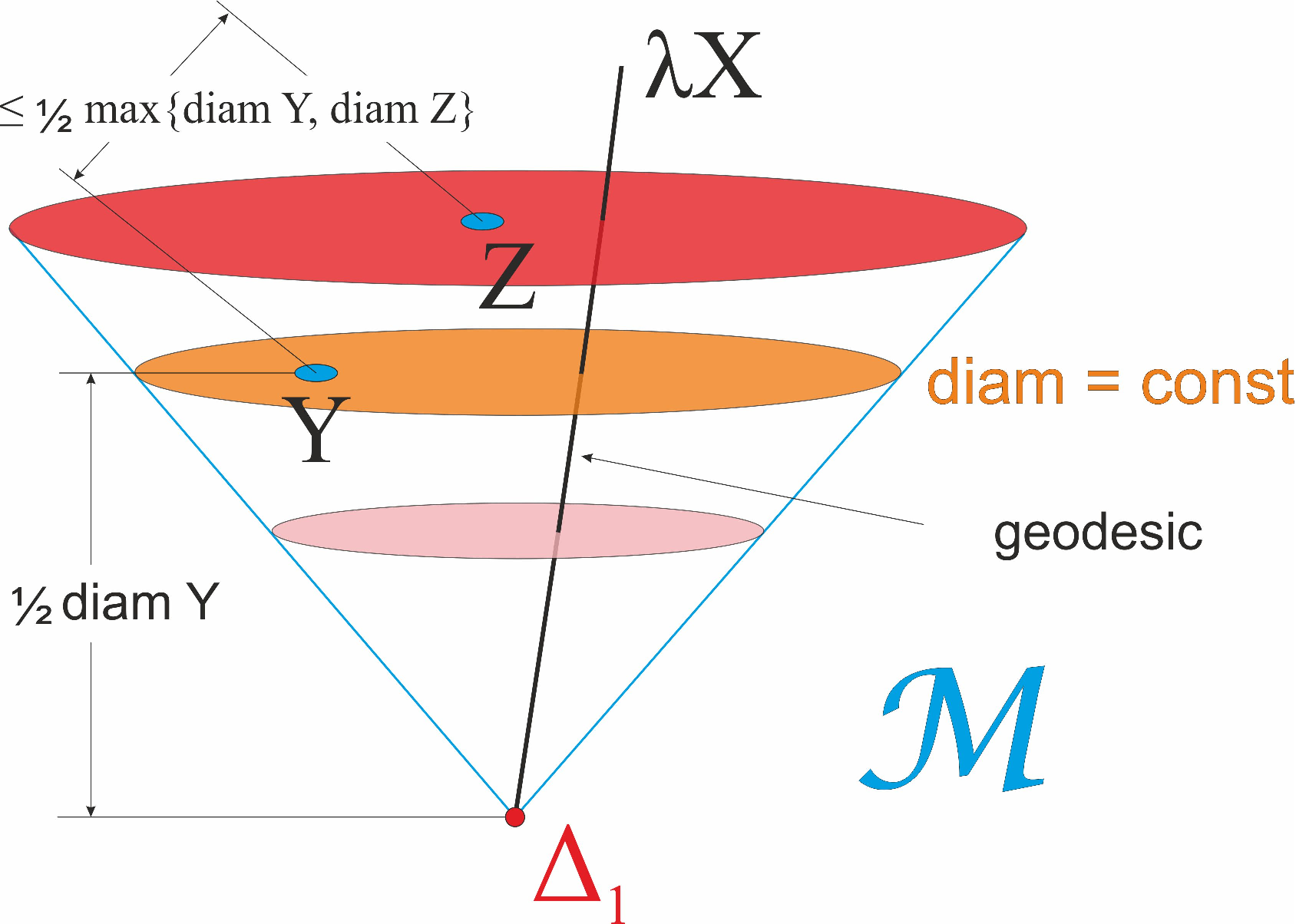}{0.35}{fig:gh-space}{The Gromov--Hausdorff space: some general properties.}

\subsection{Irreducible Correspondences}
Let $X$ and $Y$ be finite metric spaces, then the set $\cR(X,Y)$ is finite, thus there exists an $R\in\cR(X,Y)$ such that $d_{GH}(X,Y)=\frac12\dis R$. Every such correspondence $R$ we call \emph{optimal}. Notice that optimal correspondences always exist for any compact metric spaces $X$ and $Y$, see~\cite{Memoli} and~\cite{IvaIliadisTuz}. By $\cR_\opt(X,Y)$  we denote the set of all optimal correspondences between $X$ and $Y$. Thus, the following result holds.

\begin{prop}[\cite{Memoli}, \cite{IvaIliadisTuz}]
Let $X$ and $Y$ be compact metric spaces. Then $\cR_\opt(X,Y)\ne\0$.
\end{prop}

The inclusion relation generates the standard partial order on $\cR(X,Y)$, namely, $R_1\le R_2$ iff $R_1\ss R_2$. The correspondences minimal with respect to this order are called \emph{irreducible}. By $\cR^0(X,Y)$ we denote the set of all irreducible correspondences between $X$ and $Y$. It is shown in~\cite{IvaTuzIrreducible} that each $R\in\cR(X,Y)$ contains an irreducible correspondences and, thus, the following result holds.

\begin{prop}
For any metric spaces $X$ and $Y$ we have $\cR^0(X,Y)\ne\0$.
\end{prop}

The next results describes irreducible correspondences.

\begin{prop}\label{prop:decompose_inrreducible}
For each $R\in\cR^0(X,Y)$ there exist partitions $R_X=\{X_i\}_{i\in I}$ and $R_Y=\{Y_i\}_{i\in I}$ of the spaces $X$ and $Y$, respectively, such that $R=\cup_{i\in I}X_i\x Y_i$.
\end{prop}

\begin{proof}
Put $R_X=\cup_{y\in Y}\bigl\{R^{-1}(y)\bigr\}$, $R_Y=\cup_{x\in X}\bigl\{R(x)\bigr\}$, and let us show that $R_X$ and $R_Y$ are partitions. Suppose otherwise, and let, say, $R_Y$ is not a partition. Since $R$ is a correspondence, then $R_Y$ is a covering of $Y$ such that some of its elements $R(x)$ and $R(x')$ for $x\ne x'$ intersect each other, but, due to definition of $R_Y$, do not coincide. Let $y\in R(x)\cap R(x')$, then $(x,y),(x',y)\in R$. Since $R(x)\ne R(x')$, one of these sets contains an element which does not lie in the other one. To be definite, let $y'\in R(x')\sm R(x)$. Then $(x',y')\in R$, therefore, if we remove  $(x',y)$ from $R$, then we obtain a relation $\s$ such that $y\in\s(x)$ and $x'\in\s^{-1}(y')$, so $\s$ is a correspondence. The latter contradicts to irreducibility of $R$.

Thus, let us write down the partition $R_X$ in the form $\{X_i\}_{i\in I}$. Notice that for any $x,x'\in X_i$ we have $R(x)=R(x')$. Indeed, if $X_i=R^{-1}(y)$, then $R(x)$ and $R(x')$ contain $y$ and, therefore, they intersect each other. However, $R_Y$ is a partition, so we get $R(x)=R(x')$.

Choose arbitrary $i\in I$, $x\in X_i$, and put $Y_i=R(x)$. The latter definition is correct, because according to the above reasoning it does not depend on the choice of $x$. Now we show that the correspondence $\v\:X_i\mapsto Y_i$ is a bijection between $R_X$ and $R_Y$.

If $\v$ is not injective, then there exist $x,x'\in X$ belonging to different elements of the partition  $R_X$ and such that $R(x)=R(x')$. However, in this case for $y\in R(x)$ it holds $x,x'\in R^{-1}(y)\in R_X$, a contradiction.

At last, $\v$ is surjective because for any $Y_i$, $y\in Y_i$, the set $R^{-1}(y)$ is an element of the partition $R_X$. Choose an arbitrary $x\in R^{-1}(y)$. Then $R(x)\in R_Y$ contains $y$, thus $\v\bigl(R^{-1}(y)\bigr)=Y_i$.

Since for any $x,x'\in X_i$ we have $R(x)=R(x')=Y_i$, then $X_i\x Y_i\ss R$. On the other hand, since $R_X$ is a partition of $X$, then for any $x\in X$ there exists $X_i\in R_X$ such that $x\in X_i$, therefore, each $(x,y)\in R$ is contained in some $X_i\x Y_i$.
\end{proof}

\subsection{Partitions}
For any nonempty subsets $A$ and $B$ of a metric space $X$ we put
$$
|AB|'=\sup\bigl\{|ab|:a\in A,\,b\in B\bigr\}.
$$

If $D=\{X_i\}_{i\in I}$ is a partition of a metric space $X$, then we define the \emph{diameter of this partition\/} as follows: $\diam D=\sup_{i\in I}\diam X_i$. We also put
$$
\a(D)=\inf\bigl\{|X_iX_j|:i,j\in I,\,i\ne j\bigr\},\ \ \b(D)=\sup\bigl\{|X_iX_j|':i,j\in I,\,i\ne j\bigr\}.
$$

The next result can be easily obtained from the definition of distortion and from Proposition~\ref{prop:decompose_inrreducible}.

\begin{prop}\label{prop:disRforPartition}
Let $X$ and $Y$ be arbitrary metric spaces, $D_X=\{X_i\}_{i\in I}$, $D_Y=\{Y_i\}_{i\in I}$ be some partitions of the spaces $X$ and $Y$, respectively, and $R=\cup_{i\in I}X_i\x Y_i\in\cR(X,Y)$. Then
\begin{multline*}
\dis R=\sup\bigl\{|X_iX_j|'-|Y_iY_j|,\,|Y_iY_j|'-|X_iX_j|:i,j\in I\bigr\}=\\
=\sup\bigl\{\diam D_X,\,\diam D_Y,\,|X_iX_j|'-|Y_iY_j|,\,|Y_iY_j|'-|X_iX_j|:i,j\in I,\,i\ne j\bigr\}\le\\ \le\max\bigl\{\diam D_X,\diam D_Y,\b(D_X)-\a(D_Y),\b(D_Y)-\a(D_X)\bigr\}.
\end{multline*}
In particular, if $R\in\cR^0(X,Y)$, then in the previous formula one can take $R_X$ and $R_Y$ from Proposition~$\ref{prop:decompose_inrreducible}$ instead of $D_X$ and $D_Y$.
\end{prop}

For a set $X$ and any $n\in\N$ by $\cD_n(X)$ we denote  the family of all partitions of the set $X$ into $n$ nonempty subsets. Notice that for $n>\#X$ we have $\cD_n(X)=\0$, and for $n=\#X$ the family $\cD_n(X)$ consists of the unique partition of $X$ into its one-element subsets.

Let $X$ be an arbitrary metric space. The next characteristic of $X$ will be used below in the proof of our main results:
$$
d_n(X)=
\begin{cases}
\inf\bigl\{\diam D:D\in\cD_n(X)\bigr\},&\text{if $\cD_n(X)\ne\0$},\\
\infty,&\text{if $\cD_n(X)=\0$}.
\end{cases}
$$

\begin{rk}
If $X$ is a finite metric space and $n=\#X$, then $d_n(X)=0$.
\end{rk}

\begin{rk}\label{rk:dm-monotonicity}
The function $g(n)=d_n(X)$ decreases monotonically on the set of those $n$ for which $\cD_n(X)\ne\0$.
\end{rk}

\subsection{Optimal Irreducible Correspondences}

In was proved in~\cite{IvaTuzIrreducible} that for compact metric spaces $X$ and $Y$ there always exists an optimal irreducible correspondence $R$. By $\cR_\opt^0(X,Y)$  we denote the set of all irreducible optimal cor\-res\-pon\-den\-ces between $X$ and $Y$. Thus, the following result holds.

\begin{prop}[\cite{IvaTuzIrreducible}]\label{prop:opt-irreducible}
Let $X$ and $Y$ be arbitrary compact metric spaces, then $\cR_\opt^0(X,Y)\ne\0$.
\end{prop}

\begin{cor}\label{cor:dGH-for-irreducible-corr}
Let $X$ and $Y$ be arbitrary compact metric spaces, $R\in\cR^0_\opt(X,Y)$, $R_X=\{X_i\}_{i\in I}$, $R_Y=\{Y_i\}_{i\in I}$, $R=\cup_{i\in I}X_i\x Y_i$. Then
$$
2d_{GH}(X,Y)=\sup\bigl\{\diam R_X,\,\diam R_Y,\,|X_iX_j|'-|Y_iY_j|,\,|Y_iY_j|'-|X_iX_j|:i,j\in I,\,i\ne j\bigr\}.
$$
\end{cor}

\subsection{Distances to Simplexes}

A metric space $X$ we call a \emph{simplex}, if all its nonzero distances are equal to each other. Notice that a simplex $X$ is compact, iff it consists of a finite number of points. By $\D_n$  we denote the simplex consisting of $n$ points on the distance $1$ from each other. Then for $t>0$ the metric space $t\,\D_n$ is a simplex, whose points are on the distance $t$ apart from each other. Notice that $\D_1$ is the single-point metric space, and that $t\,\D_1=\D_1$ for all $t>0$. In what follows, we put $\D_n=\{1,\ldots,n\}$ for convenience.

For any metric space $X$, $n\le\#X$, and $D=\{X_1,\ldots,X_n\}\in\cD_n(X)$ we put $R_D=\sqcup\,\bigl(\{i\}\x X_i\bigr)\in\cR(t\,\D_n,X)$. Let us note that if $D'\in\cD_n(X)$ differs from $D$ by a renumbering of its elements, then $\dis R_D=\dis R_{D'}$.

\begin{prop}[\cite{IvaTuzGeometryGHandSimplexes}]\label{prop:disRD}
Let $X$ be an arbitrary metric space and $n\in\N$, $n\le\#X$. Then for any $t>0$ and $D\in\cD_n(X)$ we have
$$
\dis R_D=\max\{\diam D,\,t-\a(D),\,\b(D)-t\}.
$$
\end{prop}

\begin{prop}[\cite{IvaTuzGeometryGHandSimplexes}]\label{prop:m_less_than_n_corresp}
Let $X$ be a compact metric space. Then for each $n\in\N$, $n\le\#X$, and $t>0$ there exists some $R\in\cR_{\opt}(t\,\D_n,X)$ such that the family $\bigl\{R(i)\bigr\}$ is a partition of the space $X$. In particular, if $n=\#X$, then this $R$ can be chosen among bijections.
\end{prop}

The next result follows from Propositions~\ref{prop:disRD} and~\ref{prop:m_less_than_n_corresp}.

\begin{cor}\label{cor:GH-dist-alpha-beta}
Let $X$ be a compact metric space and $n\in\N$, $n\le\#X$. Then for any $t>0$ we have
$$
2d_{GH}(t\,\D_n,X)=\inf\Bigl\{\max\bigl(\diam D,\,t-\a(D),\,\b(D)-t\bigr):D\in\cD_n(X)\Bigr\}.
$$
\end{cor}

\begin{prop}[\cite{IvaTuzGeometryGHandSimplexes}]\label{prop:dist-n-simplex-finite}
Let $X$ be a finite metric space, $m=\#X$, $n\in\N$, $t>0$. Denote by $a\le b$ the first and the second smallest distances between different points of the space $X$ \(if they are defined\/\). Then
$$
2d_{GH}(t\,\D_n,X)=
\begin{cases}
\max\{t,\,\diam X-t\}&\text{for $m<n$},\\
\max\{t-a,\,\diam X-t\}&\text{for $m=n\ge2$},\\
\max\{a,\,t-b,\,\diam X-t\}&\text{for $m=n+1\ge3$},\\
\max\{d_n(X),\,\diam X-t\}&\text{for $m\ge n$ and $\diam X\ge2t$}.
\end{cases}
$$
Moreover, for $m=n+1$ there exists an optimal correspondence sending some point of the simplex to a pair of the closest points of $X$, and forming a bijection between the remaining points.
\end{prop}

Proposition~\ref{prop:dist-n-simplex-finite} implies an explicit formula for the Gromov--Hausdorff distance between simplexes.

\begin{cor}\label{cor:simplexes-distance}
For integer $p,q\ge2$ and real $t,s>0$ we have
$$
2d_{GH}(t\,\D_p,s\,\D_q)=
\begin{cases}
|t-s|& \text{for $p=q$},\\
\max\{t,s-t\}& \text{for $p>q$},\\
\max\{s,t-s\}& \text{for $p<q$}.
\end{cases}
$$
In particular, if $p\ne q$, then $2d_{GH}(t\,\D_p,s\,\D_q)\ge\min\{t,s\}$.
\end{cor}

\begin{prop}\label{prop:GH-more-than-simplex}
Let $X$ be a metric space containing a subspace isometric to $t\,\D_n$, $n\ge2$, and suppose that $M$ is a finite metric space, $\#M\le n-1$. Then $2d_{GH}(X,M)\ge t$. If $\diam X=t$ and $\diam M\le t$, then $2d_{GH}(X,M)=t$.
\end{prop}

\begin{proof}
Indeed, denote by $C=\{c_1,\ldots,c_n\}$ a subspace of $X$ isometric to $t\,\D_n$, then for any $R\in\cR(X,M)$ there exists $p\in M$ and distinct $c_i,\,c_j$ such that $(c_i,p),\,(c_j,p)\in R$, so $\dis R\ge t$. Since $R$ is an arbitrary correspondence, then $d_{GH}(X,M)\ge t$. If $\diam X=t$ and $\diam M\le t$, then Item~(\ref{prop:GH_simple:3}) of Proposition~\ref{prop:GH_simple} implies that $2d_{GH}(X,M)\le t$.
\end{proof}

\section{Isometries of Metric Spaces}
\markright{\thesection.~Isometries of Metric Spaces}

In this section we work out some technique suitable for description of metric spaces isometries. The main attention we pay to self-isometries.

\subsection{Operations with Invariant Subsets}
Let $X$ be a metric space and $f\:X\to X$ be an isometry. By $\cP^f(X)$ we denote  the set of all subsets of $X$ invariant with respect to $f$, namely, $\cP^f(X)=\bigl\{A\ss X:f(A)=A\bigr\}$. The next statement is evident.

\begin{prop}\label{prop:general_prop_isom}
The family $\cP^f(X)$ contains $X$, $\0$, and it is invariant under the operations of union, intersection, and taking complement. Besides that, if $A\in\cP^f(X)$, then for any $r>0$ it holds $U_r(A)\in\cP^f(X)$, and for any $r\ge0$ we have $B_r(A),S_r(A)\in\cP^f(X)$.
\end{prop}

\subsection{Isometries of Finite Pointed Spaces}
A set $X$ we call \emph{pointed} if one of its elements is marked. More formally, a \emph{pointed set\/} is a pair $(X,x)$, where $x\in X$. For a pointed set $X$ by $p(X)$  we denote its marked point $x$. Two pointed metric spaces $X$ and $Y$ are called \emph{$p$-isometric} if there exists an isometry $f\:X\to Y$ such that $p(Y)=f\bigl(p(X)\bigr)$. For a pointed metric space $X$ by $\Gr(X)$ we denote  the class of all metric spaces that are $p$-isometric to $X$.  By $\cM_*$ we denote the set of the classes of $p$-isometric pointed compact metric spaces. Thus, if $X$ is a pointed compact metric space, then $\Gr(X)\in\cM_*$.

Let $X$ be an arbitrary metric space and $x\in X$. For any $n\in\N$ by $\cP_n(x)$ we denote the set of all pointed $n$-point subspaces $Z\ss X$ containing $x$ as a marked point, i.e., such that $p(Z)=x$. Also, we define $\cM_*(X,x,n)\ss\cM_*$ to be $\Gr\bigl(\cP_n(x)\bigr)$. The following statement is evident.

\begin{prop}\label{prop:punct_isom_and_GH}
Let $f\:X\to Y$ be an isometry of metric spaces, then for any $n\in\N$ and any point $x\in X$ we have $\cM_*(X,x,n)=\cM_*\bigl(Y,f(x),n\bigr)$. In particular, each isometry $f\:X\to X$ is invariant on the level sets of the mapping $x\mapsto\cM_*(X,x,n)$.
\end{prop}

A triple $\{A,B,C\}$ of different points of a metric space $X$ we call a \emph{triangle\/}, denote by $ABC$, and write $ABC\ss X$. For such triangles we use school geometry terminology.

\begin{prop}\label{prop:punct_triangles}
Let $P$ and $Q$ be distinct points of a metric space $X$. Suppose that for each triangle $PBC\ss X$ its side $BC$ cannot be the longest one, but among the triangles $QBC\ss X$ there exists one, whose longest  side is $BC$. Then $\cM_*(X,P,3)\ne\cM_*(X,Q,3)$.
\end{prop}

\subsection{Equidistant Points Families}
For any points $P$ and $Q$ of a metric space $X$ by $\Mid(X,P,Q)$ we denote the set of all points $A\in X$ such that $|AP|=|AQ|$.

The next statement is evident.

\begin{prop}\label{prop:isom_and_midsets}
Let $f\:X\to Y$ be an arbitrary isometry of metric spaces, then for any $P,Q\in X$ it holds $f\bigl(\Mid(X,P,Q)\bigr)=\Mid\bigl(Y,f(P),f(Q)\bigr)$. In particular, each isometry $f\:X\to X$ preserving the points $P,Q\in X$ takes $\Mid(X,P,Q)$ onto itself.
\end{prop}

\section{Invariant Subspaces in $\cM$}
\markright{\thesection.~Invariant Subspaces in $\cM$}

Several ideas concerning the invariance of some subspaces of $\cM$ under a self-isometry of $\cM$ are taken from~\cite{blog}.

\subsection{Invariance of $\D_1$}

\begin{thm}\label{thm:isometric-D1}
For any $A\in\cM$, $A\ne\D_1$, it holds
$$
\cM_*(\cM,A,3)\ne\cM_*(\cM,\,\D_1,3).
$$
\end{thm}

\begin{proof}
By Items~(\ref{prop:GH_simple:1}) and~(\ref{prop:GH_simple:3}) of Proposition~\ref{prop:GH_simple}, for any $B,C\in\cM$ we have
$$
d_{GH}(B,C)\le\max\bigl\{d_{GH}(B,\D_1),\,d_{GH}(\D_1,C)\bigr\},
$$
thus, if $X=\D_1$, then in each triangle $XBC\ss\cM$ the side $BC$ cannot be the longest one.

If $A\in\cM$, $A\ne\D_1$, then, by item~(\ref{prop:GH_simple:4}) of Proposition~\ref{prop:GH_simple}, the curve $\g(t)=t\,A$, $t\in[1/2,2]$, is a shortest geodesic for which $A$ is an interior point. Therefore, by Theorem~\ref{thm:general_props}, for $B=\g(1/2)$ and $C=\g(2)$ we have $d_{GH}(B,C)=d_{GH}(B,A)+d_{GH}(A,C)$, thus, in such triangle $ABC$ the side $BC$ is the longest one. It remains to apply Proposition~\ref{prop:punct_triangles}.
\end{proof}

The next result follows immediately  from Theorem~\ref{thm:isometric-D1}, Proposition~\ref{prop:punct_isom_and_GH}, and Item~(\ref{prop:GH_simple:1}) of Propo\-sition~\ref{prop:GH_simple}.

\begin{cor}\label{cor:fD1}
Let $f\:\cM\to\cM$ be an arbitrary isometry, then $f(\D_1)=\D_1$. In particular, for any $X\in\cM$ we have $\diam f(X)=\diam X$.
\end{cor}

\subsection{Invariance of $t\,\D_n$, $n\ge2$}

Denote by $\cM^t$ the set of all $A\in\cM$ such that $\diam A\le t$. In other words, $\cM^t$ is a ball in $\cM$ of radius $t/2$ centered at $\D_1$.

\begin{thm}\label{thm:isometric-Dn}
For $t>0$ and any $A\in\cM^t$, $A\ne t\,\D_n$, $n=1,2,\ldots$, it holds
$$
\cM_*(\cM^t,A,3)\ne\cM_*(\cM^t,t\,\D_n,3).
$$
\end{thm}

\begin{proof}
Similarly with the proof of Theorem~\ref{thm:isometric-D1}, (1) for each triangle $XBC\ss\cM^t$ with $X=t\D_n$ we show that the side $BC$ cannot be longer than the remaining sides; (2) we prove that for each $A\in\cM^t$, $A\ne t\D_n$, there exists a triangle $ABC\ss\cM^t$ such that the side $BC$ is longer than the remaining two sides; after that we apply Proposition~\ref{prop:punct_triangles}.

(1) Since the case $X=t\,\D_1=\D_1$ is already considered in Theorem~\ref{thm:isometric-D1}, we straightly pass to the case $n>1$.

Suppose otherwise, i.e., that for some $n$ there is a triangle $XBC\ss\cM^t$ such that $BC$ is its longest side. By Item~(\ref{prop:GH_simple:3}) of Proposition~\ref{prop:GH_simple}, we have $2d_{GH}(B,C)\le t$, therefore, $2d_{GH}(X,B)<t$ and $2d_{GH}(X,C)<t$.

\begin{lem}
Under the assumptions made above, we have $\#B\ge n$ and $\#C\ge n$.
\end{lem}

\begin{proof}
Suppose otherwise, and let, say, $\#B<n$, then, by Proposition~\ref{prop:dist-n-simplex-finite}, we have $2d_{GH}(X,B)=\max\{t,\diam B-t\}=t$, however, $2d_{GH}(X,B)<t$, a contradiction.
\end{proof}

Further, by Proposition~\ref{prop:m_less_than_n_corresp}, there exist $R\in\cR_{\opt}(t\,\D_n,B)$ and $S\in\cR_{\opt}(t\,\D_n,C)$ such that $D=\{R(i)\}$ and $E=\{S(i)\}$ are partitions of the spaces $B$ and $C$, respectively. Put $B_i=R(i)$, $C_i=S(i)$, $T=\cup_{i=1}^nB_i\x C_i$, then $T\in\cR(B,C)$ and, by Proposition~\ref{prop:disRforPartition}, we have
$$
\dis T\le\max\bigl\{\diam D,\diam E,\b(D)-\a(E),\b(E)-\a(D)\bigr\}.
$$
By Proposition~\ref{prop:disRD}, it holds
$$
\dis R=\max\{\diam D,\,t-\a(D),\,\b(D)-t\},\ \ \dis S=\max\{\diam E,\,t-\a(E),\,\b(E)-t\},
$$
therefore, $\max\{\diam D,\,t-\a(D)\}\le\dis R$ and $\max\{\diam E,\,t-\a(E)\}\le\dis S$. Since $\diam B\le t$ and $\diam C\le t$, then $\b(D)\le t$ and $\b(E)\le t$, thus
\begin{multline*}
2d_{GH}(B,C)\le\dis T\le\max\bigl\{\diam D,\diam E,t-\a(E),t-\a(D)\bigr\}\le\\ \le\max\{\dis R,\,\dis S\}=\max\bigl\{2d_{GH}(X,B),\,2d_{GH}(X,C)\bigr\},
\end{multline*}
and so $BC$ cannot be the longest side of the triangle $XBC$, a contradiction.

(2) If $\diam A<t$, then, by Item~(\ref{prop:GH_simple:4}) of Proposition~\ref{prop:GH_simple}, the curve $\g(s)=s\,A$, $s\in[1/2,t/\diam A]$, is a shortest geodesic belonging to $\cM^t$, because $\diam\g(s)\le(t/\diam A)\diam A=t$. Besides that, $A$ is an interior point of the curve $\g$. Thus, by Theorem~\ref{thm:general_props}, for $B=\g(1/2)$ and $C=\g(t/\diam A)$ we have $d_{GH}(B,C)=d_{GH}(B,A)+d_{GH}(A,C)$, therefore, in such a triangle $ABC$ the side $BC$ is the longest one.

Now, let $\diam A=t$, then $|xx'|\le t$ for all $x,x'\in A$, and for some pair of points the equality holds, but for some other pair we have inequality, because $A\neq t\D_n$. In particular, $\#A\ge 3$.

Suppose at first that $A$ is a finite metric space consisting of $m\ge 3$ elements, and let $a\le b$ be the two smallest distances between different points of the space $A$. Put $B=t\,\D_m$ and $C=t\,\D_{m-1}$. Then, by Proposition~\ref{prop:dist-n-simplex-finite}, we have $2d_{GH}(B,C)=t$ and $2d_{GH}(A,C)=\max\{a,t-b,\diam A-t\}$. Since $a<t$ is the least nonzero distance in $A$, $b>0$, and $\diam A\le t$, then $2d_{GH}(A,C)<t$.

Further, Proposition~\ref{prop:dist-n-simplex-finite} implies that $2d_{GH}(A,B)=\max\{t-a,\diam A-t\}$. Since $a>0$ and $\diam A=t$, then $2d_{GH}(A,B)<t$. Thus, in the case under consideration we have
$$
2\max\bigl\{d_{GH}(A,B),\,d_{GH}(A,C)\bigr\}<t=2d_{GH}(B,C),
$$
so $BC$ is the longest side of the triangle $ABC$.

Suppose now that $A$ is infinite. Fix an arbitrary $\e\in(0,t/4)$, choose a finite $\e$-net $\{b_1,\ldots,b_{m-1}\}$ in $A$, and take it as the space $B$. Then, by Proposition~\ref{prop:e-net-GH}, we have $d_{GH}(A,B)\le\e<t/4$.

Let $b_m\in A$ be an arbitrary point distinct from the chosen $b_i$. Put $A_i=B_\e(b_i)$, $i=1,\ldots,m$. Then $A=\cup_{i=1}^mA_i$ and $\diam A_i<t/2$ for all $i$.

To construct $C$ we take a set $\{c_1,\ldots,c_{2m}\}$ and define the distances on it as follows: $|c_ic_j|=t$ for all $1\le i<j\le m$, and all the remaining distances are equal to $t/2$. Since the subspace $\{c_1,\ldots,c_m\}\ss C$ is isometric to $t\,\D_m$, and the diameters of $B$ and $C$ are at most $t$, then, by Proposition~\ref{prop:GH-more-than-simplex}, we have $2d_{GH}(B,C)=t$.

Consider the following correspondence $R\in\cR(A,C)$:
$$
R=\{(b_i,c_i)\}_{i=1}^m\cup\bigl(A_1\x\{c_{m+1}\}\bigr)\cup\cdots\cup\bigl(A_m\x\{c_{2m}\}\bigr).
$$
It is easy to see that $\dis R<t$, thus $2d_{GH}(A,C)<t$, therefore, $BC$ is the longest side of the triangle $ABC$. Theorem is proved.
\end{proof}

\begin{cor}\label{cor:fDn}
Let $f\:\cM\to\cM$ be an arbitrary isometry, then for any integer $n\ge2$ and real $t>0$ we have $f(t\,\D_n)=t\,\D_n$.
\end{cor}

\begin{proof}
By Proposition~\ref{prop:punct_isom_and_GH}, each isometry of the space $\cM$ preserves the family $\{t\,\D_n\}_{n=1}^\infty$, i.e., for each integer $n\ge2$ there exists $m\ge2$ such that $f(t\,\D_n)=t\,\D_m$. We have to show that $m=n$. To do that, we prove a number of auxiliary statements.

\begin{lem}\label{lem:rays-stable}
Suppose that for some $p,q\ge 2$ and $t>0$ we have $f(t\,\D_p)=t\,\D_q$, then $f(s\,\D_p)=s\,\D_q$ for all $s>0$.
\end{lem}

\begin{proof}
Suppose otherwise, i.e., that for some $s>0$ it holds $f(s\,\D_p)=s\,\D_r$, $r\ne q$. Since $f$ is isometric by Corollary~\ref{cor:simplexes-distance} we have
$$
|t-s|=2d_{GH}(t\,\D_p,s\,\D_p)=2d_{GH}\bigl(f(t\,\D_p),f(s\,\D_p)\bigr)=2d_{GH}(t\,\D_q,s\,\D_r)\ge\min\{t,s\},
$$
that does not hold for $s\in(t-t/2,t+t/2)$. This implies that the function $t\mapsto q$ is locally constant. Since each ray is connected, we get that this function is constant.
\end{proof}

\begin{lem}
Suppose that for some $p,q\ge 2$ it holds $f(\D_p)=\D_q$. Then for each $i<p$, $f(\D_i)=\D_j$, we have $j<q$.
\end{lem}

\begin{proof}
Indeed, suppose otherwise, i.e., that $j>q$ (the case $j=q$ is impossible, because $f$ is bijective). Then, by Corollary~\ref{cor:simplexes-distance} and Lemma~\ref{lem:rays-stable}, for any $t,s>0$ we have
$$
\max\{t,s-t\}=2d_{GH}(t\,\D_p,s\,\D_i)=2d_{GH}\bigl(f(t\,\D_p),f(s\,\D_i)\bigr)= 2d_{GH}(t\,\D_q,s\,\D_j)=\max\{s,t-s\}.
$$
To get a contradiction, we put $s=t/3$.
\end{proof}

Let us return to the proof that $m=n$. Suppose otherwise. Without loss of generality, we assume that $m<n$ (otherwise we consider $f^{-1}$). However, in this case the mapping $f$ takes the simplexes $\D_i$, $1<i<n$, to the simplexes $\D_j$, $1<j<m$, and distinct $i$ have to correspond to distinct $j$, a contradiction.
\end{proof}

In fact, we have shown that the unique ``corner'' points of the ball with the center at the single-point metric space is this space itself, together with the simplexes belonging to the boundary sphere, see Figure~\ref{fig:gh-space-deltas}.

\ig{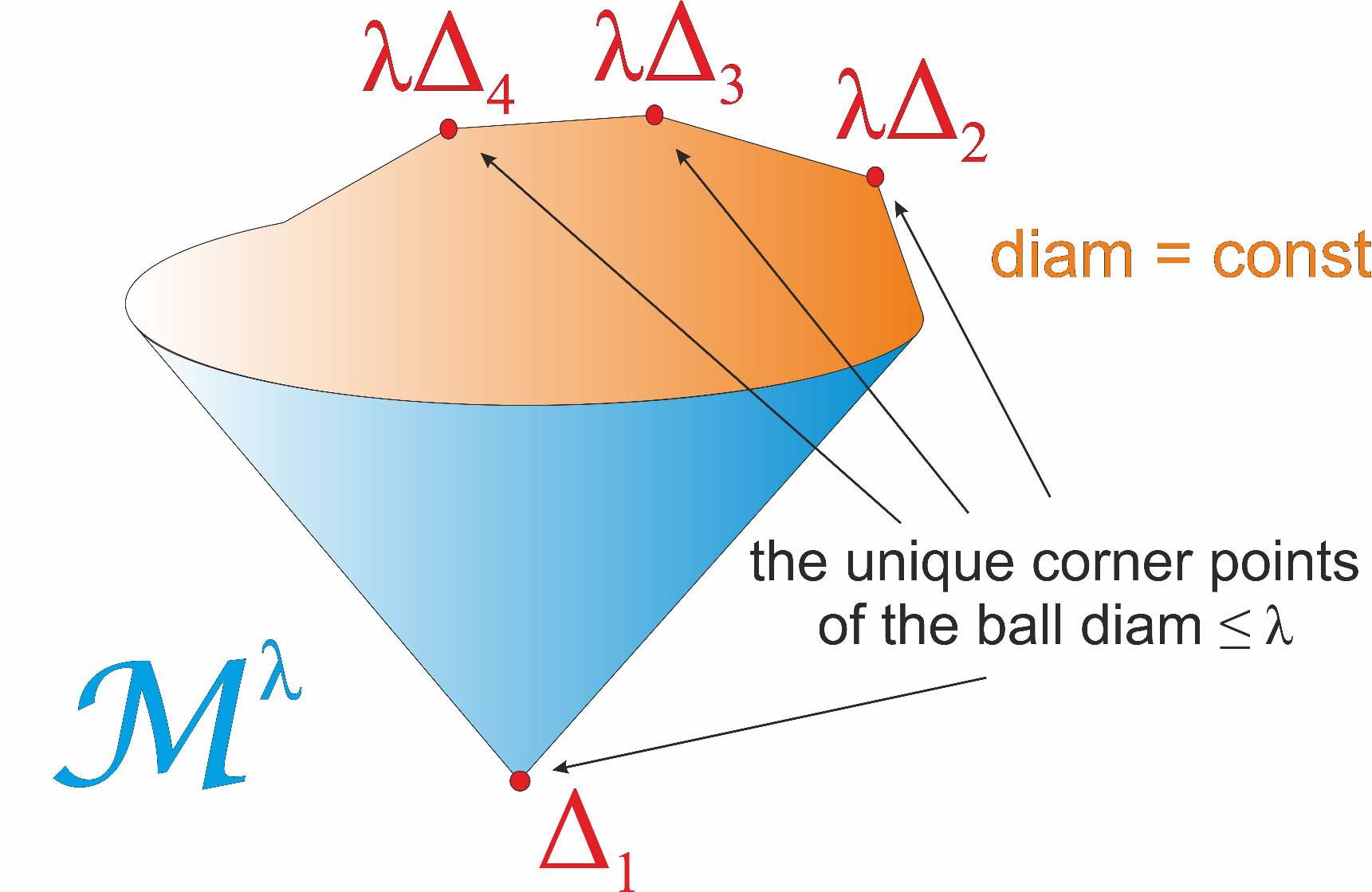}{0.28}{fig:gh-space-deltas}{``Corner'' points of the ball with the center at the single-point metric space.}

\subsection{Invariance of the Family of Finite Spaces}

For any integer $n\ge2$ and real $t>0$ put (see Figure~\ref{fig:gh-space-mid})
$$
\cB_n(t)=\Mid(\cM,\D_1,t\,\D_n)\cap\{B\in\cM:\diam B\ge2t\}.
$$

\ig{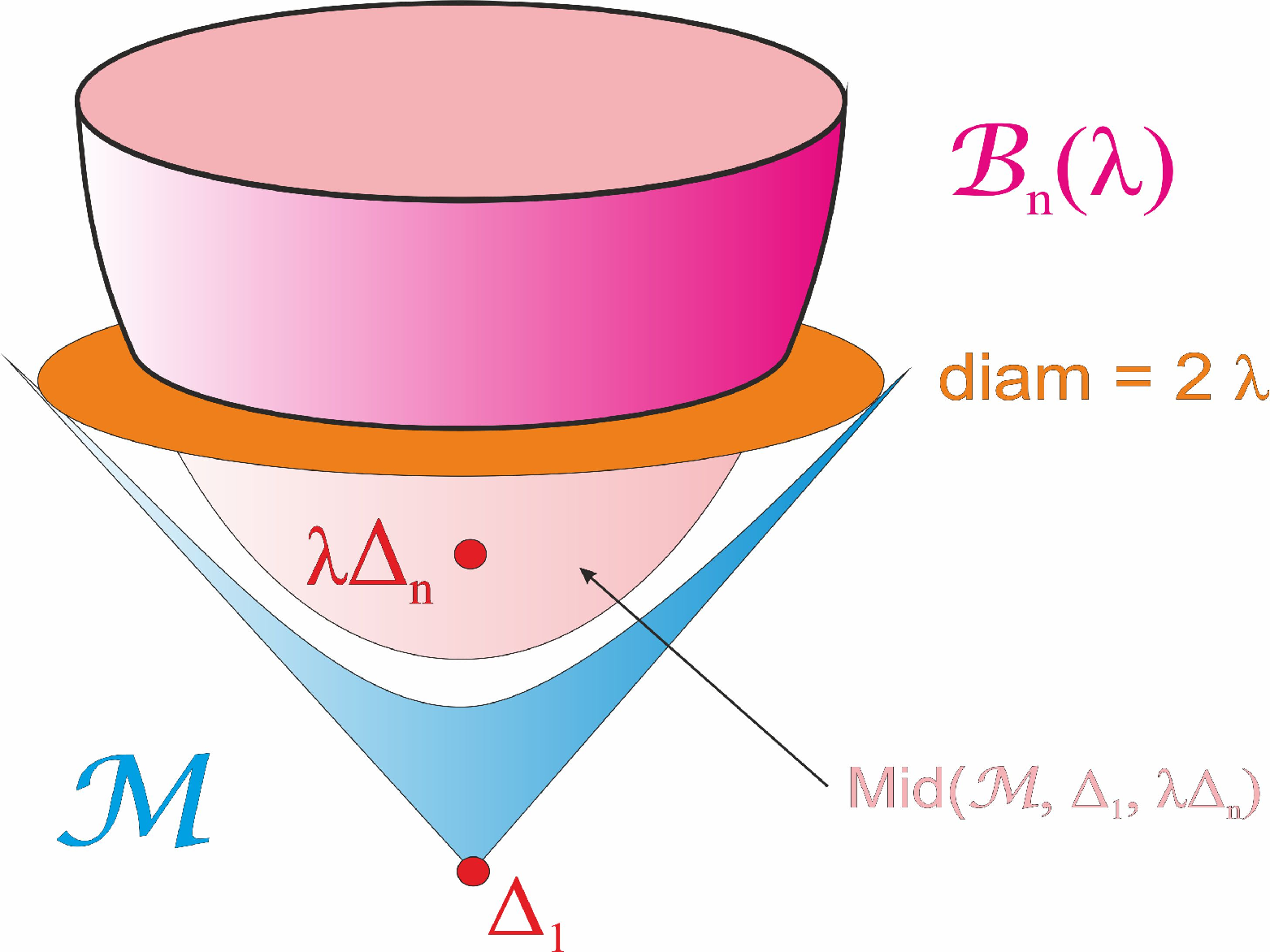}{0.28}{fig:gh-space-mid}{Illustration for $\cB_n(t)$.}

\begin{prop}\label{prop:cBn}
For each integer $n\ge2$ and real $t>0$ the following statements hold\/\rom:
\begin{enumerate}
\item\label{prop:cBn:1} let $f\:\cM\to\cM$ be an arbitrary isometry, then $f\bigl(\cB_n(t)\bigr)=\cB_n(t)$\rom;
\item\label{prop:cBn:2} for each real $s\ge2t$ and integer $m>n$ we have $s\,\D_m\in\cB_n(t)$, in particular, $\cB_n(t)\ne\0$\rom;
\item\label{prop:cBn:3} for any $B\in\cB_n(t)$ we have $\#B>n$\rom;
\item\label{prop:cBn:4} for any $B\in\cB_n(t)$ we have $d_n(B)=\diam B$.
\end{enumerate}
\end{prop}

\begin{proof}
(\ref{prop:cBn:1}) This immediately follows from Proposition~\ref{prop:isom_and_midsets},  Corollary~\ref{cor:fD1}, and Corollary~\ref{cor:fDn}.

(\ref{prop:cBn:2}) Indeed, by Corollary~\ref{cor:simplexes-distance} and Item~(\ref{prop:GH_simple:1}) of Proposition~\ref{prop:GH_simple}, we get
$$
2d_{GH}(t\,\D_n,s\,\D_m)=\max\{s,t-s\}=s=\diam(s\,\D_m)=2d_{GH}(\D_1,s\,\D_m).
$$

(\ref{prop:cBn:3}) Suppose otherwise, i.e., that $\#B\le n$. Denote by $a$ the smallest distance between different points of $B$. Then, by Proposition~\ref{prop:dist-n-simplex-finite} and definition of $\cB_n(t)$, we have
\begin{align*}
&2d_{GH}(t\,\D_n,B)=\max\{t,\diam B-t\}=\diam B-t & \text{for $\#B<n$},\\
&2d_{GH}(t\,\D_n,B)=\max\bigl\{t-a,\,\diam B-t\bigr\}=\diam B-t & \text{for $\#B=n$}.
\end{align*}
However, $2d_{GH}(t\,\D_n,B)=2d_{GH}(\D_1,B)=\diam B$, a contradiction.

(\ref{prop:cBn:4}) Since, by Item~(\ref{prop:cBn:3}), $\#B>n$, and also because $\diam B\ge2t$, we can use Proposition~\ref{prop:dist-n-simplex-finite} which implies that
$$
\diam B=2d_{GH}(t\,\D_n,B)=\max\bigl\{d_n(B),\,\diam B-t\bigr\},
$$
thus, $d_n(B)=\diam B$.
\end{proof}

For an integer $n\ge2$ and real $t>0$ put
$$
\cF_n(t)=\bigl\{A\in\cM:\diam A=t\ \text{and}\ 2d_{GH}(A,B)=2d_{GH}(\D_1,B)=\diam B\ \text{for all $B\in\cB_n(t)$}\bigr\}.
$$

\begin{rk}
The set $\cF_n(t)$ contains the simplex $t\,\D_n$ of diameter $t$, and it consists of all $A$ of the same diameter $t$ and such that they can play the role of the $t\,\D_n$ in definition of $\cB_n(t)$, see Figure~\ref{fig:gh-space-mid-focus}.
\end{rk}

\ig{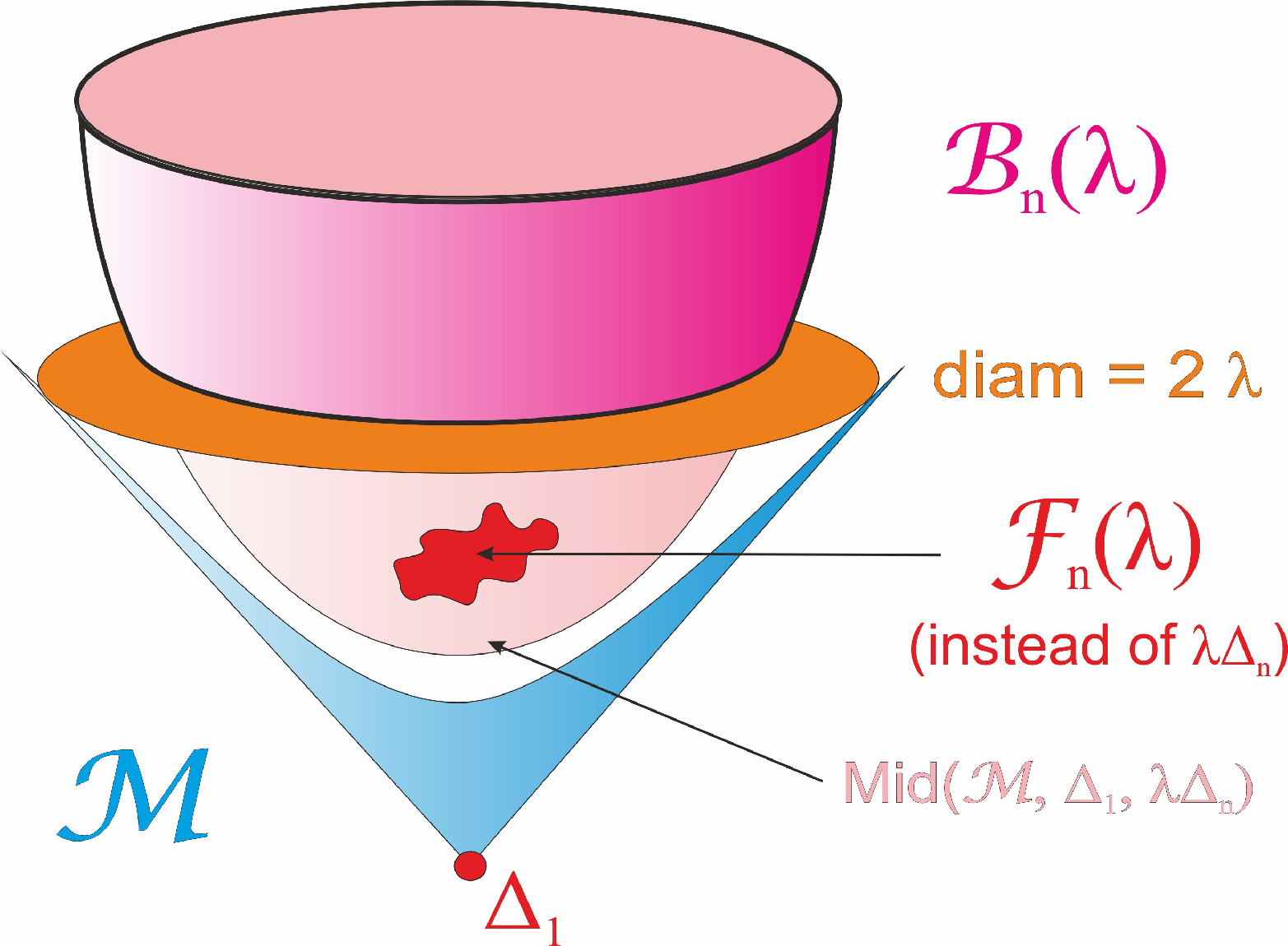}{0.28}{fig:gh-space-mid-focus}{Illustration for $\cF_n(t)$.}

\begin{prop}\label{prop:invar-cFn(t)}
The set $\cF_n(t)$ is invariant under each isometry $f\:\cM\to\cM$.
\end{prop}

\begin{proof}
For any $r>0$ put $\cB_n(t,r)=\cB_n(t)\cap S_r(\D_1)$, then $\cB_n(t,r)$ is $f$-invariant by Item~(\ref{prop:cBn:1}) of Proposition~\ref{prop:cBn}, Corollary~\ref{cor:fD1}, and Proposition~\ref{prop:general_prop_isom}. By Proposition~\ref{prop:general_prop_isom}, the equidistant $S_d\bigl(\cB_n(t,r)\bigr)$ is $f$-invariant for each $d\ge0$ as well. It remains to note that $\cF_n(t)$ equals to the intersection of $f$-invariant sets $S_{t/2}(\D_1)$ and $S_{r/2}\bigl(\cB_n(t,r)\bigr)$ over all $r\ge2t$, and apply Proposition~\ref{prop:general_prop_isom} again.
\end{proof}

\begin{thm}\label{thm:cFn-le-n}
For any integer $n\ge2$ and real $t>0$ the set $\cF_n(t)$ coincides with the set of all finite metric spaces of diameter $t$ consisting of at most $n$ points.
\end{thm}

\begin{proof}
(1) Let us show that each at most $n$-points metric space $A$ of diameter $t$ belongs to $\cF_n(t)$. To do that, take an arbitrary $B\in\cB_n(t)$ and verify that $2d_{GH}(A,B)=\diam B$.

By Proposition~\ref{prop:opt-irreducible}, there exists $R\in\cR_{\opt}^0(A,B)$. By Proposition~\ref{prop:decompose_inrreducible}, the family $R_B=\cup_{a\in A}\bigl\{R(a)\bigr\}$ is a partition of $B$ consisting of at most $n$ elements, i.e., $R_B\in\cD_m(B)$ for some $m\le n$. By Remark~\ref{rk:dm-monotonicity}, $\diam R_B\ge d_m(B)\ge d_n(B)$. By Corollary~\ref{cor:dGH-for-irreducible-corr}, we have $2d_{GH}(A,B)\ge\diam R_B$, therefore, taking into account Item~(\ref{prop:cBn:4}) of Proposition~\ref{prop:cBn}, we get $2d_{GH}(A,B)\ge d_n(B)=\diam B$. Since $\diam A<\diam B$, then, by Item~(\ref{prop:GH_simple:3}) of Proposition~\ref{prop:GH_simple}, it holds
$$
\diam B=\max\{\diam A,\diam B\}\ge2d_{GH}(A,B)\ge\diam B,
$$
thus, $2d_{GH}(A,B)=\diam B$.

(2) Now, let us show that if $A\in\cF_n(t)$, then $\#A\le n$.

Suppose otherwise, i.e., that $\#A>n$. Put $\e=t/3$ and choose a finite $\e$-net $S=\{a_1,\ldots,a_m\}$ in $A$ consisting of $m\ge n+1$ points. Let $A_i=B_{\e}(a_i)$, then $\diam A_i\le2\e<t$ and $A=\cup_{i=1}^mA_i$.

Choose an arbitrary $\mu\ge2t$, put $B=\{b_1,\ldots,b_{2m}\}$, and define a metric on $B$ as follows: $|b_ib_j|=\mu$ for $1\le i<j\le m$, and all the remaining nonzero distances are equal to $\mu/2$. Clearly that $d_n(B)=\mu=\diam B\ge2t$, therefore, by Proposition~\ref{prop:dist-n-simplex-finite}, we have $2d_{GH}(t\,\D_n,B)=\max\bigl\{d_n(B),\diam B-t\bigr\}=\diam B$, thus, $B\in\cB_n(t)$.

Define $R\in\cR(A,B)$ as follows:
$$
R=\{(a_i,b_i)\}_{i=1}^m\cup\bigl(A_1\x\{b_{m+1}\}\bigr)\cup\cdots\cup\bigl(A_m\x\{b_{2m}\}\bigr),
$$
then $2d_{GH}(A,B)\le\dis R<\mu=\diam B$, a contradiction. Thus, $\#A\le n$. Theorem is proved.
\end{proof}

\begin{cor}\label{cor:invarFiniteSpaces}
Every isometry $f\:\cM\to\cM$ takes each $n$-point metric space to an $n$-point metric space of the same diameter.
\end{cor}

\begin{proof}
By Theorem~\ref{thm:cFn-le-n}, the set $\cF_n(t)$ coincides with the family of all of diameter $t>0$ consisting of at most $n$ points. By Proposition~\ref{prop:invar-cFn(t)}, the set $\cF_n(t)$ is invariant under every isometry $f\:\cM\to\cM$, thus, each $n$-point metric space $A\in\cF_n(t)$ is mapped to an at most $n$ point metric space $B=f(A)$. Suppose that $\#B<n$. Since $f^{-1}$ is an isometry of $\cM$ also, then, by the above arguments, we have $\#A=\#f^{-1}(B)<n$, a contradiction.
\end{proof}

\section{Groups Actions}
\markright{\thesection.~Groups Actions}

In what follows, we need some basic facts from the theory of group action on topological and metric spaces.

Recall that if a compact group $G$ acts continuously on a topological space $X$, then its orbits are compact subsets. Indeed, if $\theta\:G\x X\to X$ is the mapping defining this action, then $\theta$ is continuous, and the orbit $G(x)$ of a point $x$ equals $\theta\bigl(G\x\{x\}\bigr)$, therefore, it is compact as a continuous image of a compact set. If a group $G$ acts on a set $X$, then by $X/G$ we denote the corresponding set of orbits.

If $X$ is a metric space, and the group $G$ is compact, then the following non-negative symmetric function $(A,B)\mapsto|AB|$, $A,B\in X/G$, does not vanish for any $A\ne B$.

\begin{prop}\label{prop:compact-group-action}
If a compact group $G$ acts on a metric space $X$ by isometries, then the function $(A,B)\mapsto|AB|$ defined on pairs of elements of the orbit set $X/G$ is a metric.
\end{prop}

\begin{proof}
It remains to verify the triangle inequality. Since the orbits are compact subsets  in this case, then for any $A,B,C\in X/G$ there exist $a\in A$, $b_1,b_2\in B$, and $c\in C$ such that $|ab_1|=|AB|$ and $|b_2c|=|BC|$. Since $b_1$ and $b_2$ belongs to the same orbit, there exists $g\in G$ such that $g(b_1)=b_2$. However, $g\:X\to X$ is an isometry, therefore, $\bigl|g(a)g(b_1)\bigr|=|AB|$ and, thus, $|AC|\le|g(a)c|\le\bigl|g(a)g(b_1)\bigr|+|b_2c|=|AB|+|BC|$.
\end{proof}

The metric on the set $X/G$ defined in Proposition~\ref{prop:compact-group-action} is called a \emph{quotient-metric}. In what follows, speaking about the metric space $X/G$, we always mean just this quotient-metric.

\begin{prop}\label{prop:compact-group-action-stab}
Suppose that a finite group $G$ acts of a metric space $X$. Then for every point $x\in X$ the following statements hold.
\begin{enumerate}
\item\label{prop:compact-group-action-stab:1} For any $\e>0$ and any $g$ from the stabilizer $G_x$ of the point $x$ we have $g\bigl(B_\e(x)\bigr)=B_\e(x)$. Thus, for each $\e>0$ an action of the stabilizer $G_x$ of the point $x\in X$ on the neighbourhood $B_\e(x)$ is defined.
\item\label{prop:compact-group-action-stab:2} If $G\sm G_x\ne\0$, then there exists $\e>0$ such that for all $g\in G\sm G_x$ it holds $B_\e(x)\cap g\bigl(B_\e(x)\bigr)=\0$, in particular, for every point $y\in B_\e(x)$ its stabilizer $G_y$ is a subgroup of $G_x$, and also $G(y)\cap B_\e(x)=G_x(y)$.
\item\label{prop:compact-group-action-stab:3} There exists $\e>0$ such that for any $y_1,y_2\in B_\e(x)$ the distance between the orbits $G(y_1)$ and $G(y_2)$ is equal to the distance between the orbits $G_x(y_1)$ and $G_x(y_2)$.
\end{enumerate}
\end{prop}

\begin{proof}
(\ref{prop:compact-group-action-stab:1}) Since for each $g\in G_x$ we have $g(x)=x$, and $g$ is an isometry, then $g\bigl(B_\e(x)\bigr)=B_\e(x)$ for any $\e$.

(\ref{prop:compact-group-action-stab:2}) Put $Z=\bigl\{g(x):g\in G\sm G_x\bigr\}$, then $x\not\in Z$, and $Z$ is a nonempty finite set (because $G\sm G_x\ne\0$), thus $r:=|xZ|>0$. Choose an arbitrary $\e<r/2$, then for all $g\in G\sm G_x$ we have $B_\e(x)\cap g\bigl(B_\e(x)\bigr)=\0$. In particular, this implies that the stabilizer of each point $y\in B_\e(x)$ does not intersect $G\sm G_x$. Besides, for any point $y\in B_\e(x)$ and each $g\in G_x$ we have $\bigl|x\,g(y)\bigr|=\bigl|g(x)g(y)\bigr|=|xy|\le\e$, therefore, $B_\e(x)$ contains exactly that part of the orbit $G(y)$ which is generated by the elements of the stabilizer $G_x$.

(\ref{prop:compact-group-action-stab:3}) If $G_x=G$, then we can take an arbitrary $\e$.

If $G\sm G_x\ne\0$, then for $r$ from Item~(\ref{prop:compact-group-action-stab:2}), let us choose an arbitrary $\e<r/4$, then the distance between any points from $B_\e(x)$ is less than $r/2$, and the distance between any point from $B_\e(x)$ and any point from $B_\e\bigl(g(x)\bigr)$ for $g\in G\sm G_x$ is greater than $r/2$. Thus, the distance between the orbits $G(y_1)$ and $G(y_2)$, $y_1,y_2\in B_\e(x)$, is attained at those points of these orbits that both belong to a neighbourhood $B_\e\bigl(g(x)\bigr)$, and this distance is the same in each of these neighbourhoods (because $G$ acts by isometries). By Item~(\ref{prop:compact-group-action-stab:2}), all points of the orbits in consideration that belong to the ball $B_\e(x)$ form the sets $G_x(y_1)$ and $G_x(y_2)$, respectively.
\end{proof}

\begin{dfn}
Under the assumptions and notations of Proposition~\ref{prop:compact-group-action-stab}, the closed ball $B_\e(x)$ with any $\e>0$ for $G_x=G$, and with $\e<r/4$ for $G_x\ne G$, we call a \emph{canonical neighborhood of the point $x\in X$}.
\end{dfn}

\begin{cor}\label{cor:proj-is-loc-isom}
Let $G$ be a finite group acting by isometries on a metric space $X$. Choose an arbitrary point $x\in X$. Then the stabilizer $G_x$ acts on each canonical neighbourhood $B_\e(x)$, and  $\pi_{\e,x}\:B_\e(x)/G_x\to B_\e\bigl(G(x)\bigr)\ss X/G$, $\pi_{\e,x}\:G_x(y)\mapsto G(y)$, is an isometry. Further, for each $g\in G$ the neighbourhood  $B_\e\bigl(g(x)\bigr)=g\bigl(B_\e(x)\bigr)$ is canonical also, and the mapping $g$ generates an isometry $g_{\e,x}\:B_\e(x)/G_x\to B_\e\bigl(g(x)\bigr)/G_{g(x)}$, $g_{\e,x}\:G_x(y)\mapsto G_{g(x)}\big(g(y)\big)$. Besides, the mappings $g_{\e,x}$, $\pi_{\e,x}$, and $\pi_{\e,g(x)}$ are agreed with each other in the following sense\/\rom: $\pi_{\e,x}=\pi_{\e,g(x)}\c g_{\e,x}$. Thus, each mapping $\pi^{-1}_{\e,g(x)}\c\pi_{\e,x}$ is generated by the mapping $g$.
\end{cor}

In what follows we especially need a version of Corollary~\ref{cor:proj-is-loc-isom} in the situation, when the stabilizer $G_x$ is trivial (i.e., it consists of the unit element only). Since in this case $G_x(y)=\{y\}$ for any $x,y\in X$, then the mapping $\pi_{\e,x}\:G_x(y)\mapsto G(y)$ coincides with the restriction of the canonical projection $\pi\:y\mapsto G(y)$ onto the canonical neighbourhood $B_\e(x)$. Similarly, in this case, $g_{\e,x}\:B_\e(x)/G_x\to B_\e\bigl(g(x)/G_{g(x)}\bigr)$ is a mapping between the canonical neighbourhoods $B_\e(x)$ and $B_\e\bigl(g(x)\bigr)$, and it coincides with the restriction of the mapping $g$ onto the canonical neighbourhood $B_\e(x)$, thus, in this case Corollary~\ref{cor:proj-is-loc-isom} can be reformulated as follows.

\begin{cor}\label{cor:proj-is-loc-isom-triv-stab}
Let $G$ be an arbitrary finite group acting on a metric space $X$ by isometries, and let $\pi\:X\to X/G$ be the canonical projection, $\pi\:x\mapsto G(x)$. Suppose that the stabilizers of all points from $X$ are trivial. Then the restriction $\pi_{\e,x}$ of the projection $\pi$ onto each canonical neighbourhood $B_\e(x)\ss X$ of the point $x$ maps isometrically the $B_\e(x)$ onto $B_\e\bigl(G(x)\bigr)\ss X/G$. Further, for each $g\in G$ the neighbourhood $B_\e\bigl(g(x)\bigr)=g\bigl(B_\e(x)\bigr)$ is also canonical. Besides, the restriction $g_{\e,x}\:B_\e(x)\to B_\e\bigl(g(x)\bigr)$ of the mapping $g$, being isometry, is agreed with the mappings $\pi_{\e,x}$ and $\pi_{\e,g(x)}$ in the following sense\/\rom: $\pi_{\e,x}=\pi_{\e,g(x)}\c g_{\e,x}$. Thus, each mapping $\pi^{-1}_{\e,g(x)}\c\pi_{\e,x}$ coincides with the restriction of the mapping $g$ onto the canonical neighbourhood $B_\e(x)$.
\end{cor}

\section{The Canonical Local Isometry}
\markright{\thesection.~The Canonical Local Isometry}

For $n\in\N$ put $\cM_n=\{X\in\cM:\#X\le n\}$ and $\cM_{[n]}=\{X\in\cM:\#X=n\}$, then $\cM_{[1]}=\{\D_1\}$, $\cM_{[2]}$ is isometric to the positive ray on the real line, and $\cM_{[3]}$ is isometric to the set $\{(a,b,c):0<a\le b\le c\le a+b\}$ endowed with the metric generated by the $\ell_\infty$-norm: $\bigl\|(x,y,z)\bigr\|_\infty=\frac12\max\bigl\{|x|,|y|,|z|\bigr\}$ (the latter fact one can find in~\cite{IvaTuzIrreducible}).

For $N=n(n-1)/2$ by $\R^N_\infty$ we denote the arithmetic space $\R^N$ endowed with the $\ell_\infty$-norm: $\bigl\|(x^1,\ldots,x^N)\bigr\|_\infty=\frac12\max_{i=1}^N\bigl\{|x^i|\bigr\}$. The corresponding $\ell_\infty$-distance between points $x,y\in\R^N_\infty$ is denoted by $|xy|_{\infty}$.

Let $X\in\cM_{[n]}$. Enumerate the points of $X$ in an arbitrary way, then $X=\{x_i\}_{i=1}^n$, and let $\r_{ij}=\r_{ji}=|x_ix_j|$ be the components of the corresponding distance matrix $M_X$ of the space $X$. The matrix $M_X$ is uniquely determined by the vector
$$
\r_X=(\r_{12},\ldots,\r_{1n},\r_{23},\ldots,\r_{2n},\ldots,\r_{(n-1)n})\in\R^N.
$$
Notice that the set of all possible $\r_X\in\R^N$, $X\in\cM_{[n]}$, consists of all vectors with positive coordinates, which satisfy the following ``triangle inequalities'': for any pairwise distinct $1\le i,j,k\le N$ we have $\r_{ik}\le\r_{ij}+\r_{jk}$ (here, for convenience, we put $\r_{ij}=\r_{ji}$ for all $i$ and $j$). The set of all such vectors is denoted by $\cC_n$.

If one changes the numeration of points of the space $X$, i.e., if one acts by a permutation $\s\in S_n$ on $X$ by the rule $\s(x_i)=x_{\s(i)}$, then the components of the matrix $M_X$ are permuted as follows: $\r_{ij}\mapsto\s(\r_{ij}):=\r_{\s(i)\s(j)}$. By $M_{\s(X)}$ we denote  the resulting matrix, and by $\r_{\s(X)}$ the corresponding vector is denoted. Thus, an action of the group $S_n$ on $\cC_n$ is defined.

Notice that the action of the group $S_n$ on $\cC_n$ just described consists in permuting of the basis vectors of $\R^N$, therefore, this action can be naturally extended onto the entire $\R^N$, and, thus, $S_n$ generates a subgroup  $G$ of the group $S_N$ consisting of all permutations of the basis vectors of the space $\R^N$. Since the unit ball in $\R^N_\infty$ is a Euclidean cube centered at the origin, and each permutation of the coordinate vectors take this cube into itself, then the group $S_N$, together with its subgroup $G$, acts on $\R^N_\infty$ by isometries. Notice also that, generally speaking, the group $S_N$ does not preserve the cone $\cC_n$, because permutations of general type acting on the set of distances of a metric space $X$ can violate a triangle inequality.

Further, each orbit of the action of the group $G$ on $\R^N$ contains at most $n!$ points, and each \emph{regular orbit}, i.e., the one having trivial stabilizer, consists of $n!$ points exactly. A space $X$ such that the orbit of the corresponding $\r_X$ is regular, together with all the vectors $g(\r_X)$, $g\in G$, we call \emph{regular}.

Notice that $\cC_n$ is not open in $\R^N$: it contains boundary points, namely, those $\r_X$ at which some triangle inequalities degenerate to equalities. Such $X$ and the corresponding $\r_X$ we call \emph{degenerate}, and all the remaining $X$ and $\r_X$ we call \emph{non-degenerate}.

We say that a space $X\in\cM_{[n]}$ and each corresponding $\r_X\in\R^N$ are \emph{generic\/} or are \emph{in general position}, if $X$ is regular and non-degenerate. Thus, $X\in\cM_{[n]}$ is generic, iff its isometry group is trivial and all triangle inequalities hold strictly. Notice that in~\cite{IvaTuzLocalStrEmbed} by generic space we meant a few narrow class of object demanding in addition that all nonzero distances are pairwise different.

Denote by $\cC_n^g$ the subset of $\cC_n$ consisting of all vectors in general position, and by $\cM_{[n]}^g$ the corresponding subset of $\cM_{[n]}$ consisting of all spaces in general position. It is easy to see that the both $\cC_n^g$ and $\cM_{[n]}^g$ are open in $\R^N$ and in $\cM_{[n]}$, respectively; besides that, these subsets are everywhere dense in $\cC_n$ and $\cM_{[n]}$, respectively.

Define a mapping $\Pi\:\cC_n\to\cM_{[n]}\ss\cM$ as  $\Pi\bigl(\r_X\bigr)=X$. Let us investigate the properties of this mapping. As it is shown in~\cite{IvaTuzLocalStrIsom}, for a sufficiently small $\e>0$ and any $Y,Z\in B_\e(X)\ss\cM_{[n]}$ each optimal correspondence $R\in\cR(Y,Z)$ is a bijection. Therefore for such $Y$ and $Z$ it holds
$$
d_{GH}(Y,Z)=\min_{\r_Y,\r_Z}\bigl\{|\r_Y\r_Z|_\infty\bigr\}=\bigl|G(\r_Y)G(\r_Z)\bigr|_\infty,
$$
where in the right hand side of the equality the standard distance between subsets of $\R^N_\infty$ stands, i.e., the infimum (here it is the minimum) of $\R^N_\infty$-distances between their elements.

Thus, we get the following result.

\begin{prop}\label{prop:loc-iso-to-GH}
For any $X\in\cM_{[n]}$ there exists $\e>0$ such that
$$
d_{GH}(Y,Z)=\bigl|\Pi^{-1}(Y)\Pi^{-1}(Z)\bigr|_\infty
$$
for every $Y,Z\in B_\e(X)\ss\cM_{[n]}$.
\end{prop}

By Proposition~\ref{prop:compact-group-action}, the action of the group $G$ on $\cC_n$ generates a metric space $\cC_n/G$. Item~(\ref{prop:compact-group-action-stab:3}) of Proposition~\ref{prop:compact-group-action-stab} implies the following statement.

\begin{cor}\label{cor:loc-str-CnG}
For sufficiently small $\e>0$ the ball $B_\e\bigl(G(\r)\bigr)$ in the space $\cC_n/G$ is isometric to the quotient space $\bigl(B_\e(\r)\cap\cC_n\bigr)/G_\r$, where $B_\e(\r)$ is a ball in $\R^N_\infty$, and $G_\r$ is the stabilizer of the point $\r\in\cC_n$ under the group $G$ action.
\end{cor}

Combining Proposition~\ref{prop:loc-iso-to-GH} and Corollary~\ref{cor:loc-str-CnG}, we get the following result.

\begin{cor}\label{cor:ConeProjTOnPoints}
The mapping $G(\r_X)\mapsto X$ is a locally isometric homeomorphism between $\cC_n/G$ and $\cM_{[n]}$, therefore, for any $X\in\cM_{[n]}$ and any $\r\in\Pi^{-1}(X)$ there exists $\e>0$ such that the closed ball $B_\e(X)\ss\cM_{[n]}$ is isometric to $\bigl(B_\e(\r)\cap\cC_n\bigr)/G_\r$, where $B_\e(\r)$ is a ball in $\R^N_\infty$, and $G_\r$ is the stabilizer of the point $\r\in\cC_n$ under the group $G$ action.
\end{cor}

Now we consider different types of the spaces $X\in\cM_{[n]}$: a generic space, a regular degenerate space, a non-regular non-degenerate space, and, at last, a non-regular degenerate space. All the corresponding results listed below follow from Corollary~\ref{cor:ConeProjTOnPoints}.

\paragraph{Generic Spaces.}

Recall that by generic spaces we mean regular nondege\-ne\-rate spaces $X\in\cM_{[n]}$ and corresponding elements from $\cC_n$.

\begin{cor}\label{cor:LocIsomGenPos}
For each generic space $X\in\cM_{[n]}$, for all sufficiently small $\e>0$ the closed ball $B_\e(X)\ss\cM_{[n]}$ lies in $\cM_{[n]}^g$ and is isometric to the ball $B_\e(\r_X)$ in $\R^N_\infty$.
\end{cor}

\paragraph{Regular Degenerate Spaces.}

For $n\ge3$ and any $\r\in\cC_n$ by $D(\r_X)$ we denote the set of all ordered triples of different indices $(i,j,k)$, $1\le i,j,k\le n$, such that $\r_{ij}+\r_{jk}=\r_{ik}$. Notice that $\r$ is degenerate, iff $D(\r)\ne\0$. Further, for nonempty $D(\r)$ by $T(\r)$ we denote the polyhedral cone with the vertex at the origin, which is obtained as the intersection of all half-spaces in $\R^N$ defined by the inequalities $\r_{ij}+\r_{jk}-\r_{ik}\ge0$ over all $(i,j,k)\in D(\r)$. If $D(\r)=\0$, then put $T(\r)=\R^N$. Notice that for a degenerate $\r\in\R^N_\infty$ and any sufficiently small $\e>0$ we have $B_\e(\r)\cap\cC_n=B_\e(\r)\cap T(\r)$.

\begin{cor}\label{cor:LocIsomDegReg}
For each regular degenerate space $X\in\cM_{[n]}$ and for all sufficiently small $\e>0$ the closed ball $B_\e(X)\ss\cM_{[n]}$ is isometric to the intersection $B_\e(\r_X)\cap T(\r_X)$ of the ball $B_\e(\r_X)$ in $\R^N_\infty$ and the cone $T(\r_X)$ defined above.
\end{cor}

We also need the following simple property of the cone $T(\r)$.

\begin{prop}\label{prop:BeCAPTrX-nonempty-inter}
If $\r_X\in\cC_n$ is a vector corresponding to a degenerate space $X\in\cM_{[n]}$, then for any $\e>0$ the set $B_\e(\r_X)\sm\bigl(B_\e(\r_X)\cap T(\r_X)\bigr)$ has a nonempty interior.
\end{prop}

\begin{proof}
Indeed, since $X$ is a degenerate space, then $T(\r_X)$ is contained in a half-space $\Theta$ bounded by a hyperplane $\theta$ of the form $\r_{ij}+\r_{jk}-\r_{ik}=0$ passing through $X$. Since the both cube $B_\e(\r_X)$ and hyperplane $\theta$ are centrally symmetric with respect to $X$, then $B_\e(\r_X)\sm(B_\e(\r_X)\cap\Theta)$ contains interior points. It remains to note that $T(\r_X)\ss\Theta$.
\end{proof}

\paragraph{Non-regular Non-degenerate Spaces.}

\begin{cor}\label{cor:LocIsomNonDegNonReg}
For each non-regular non-degenerate space $X\in\cM_{[n]}$, for a sufficiently small $\e>0$ the closed ball $B_\e(X)\ss\cM_{[n]}$ is isometric to the space $B_\e(\r_X)/G_{\r_X}$ obtained from the ball $B_\e(\r_X)$ in $\R^N_\infty$ by factorisation over action of the stabilizer $G_{\r_X}$ of the point $\r_X$.
\end{cor}

\paragraph{Non-regular Degenerate Spaces.}

Now, let $X\in\cM_{[n]}$ be a non-regular degenerate space, then the stabilizer $G_{\r_X}$ is nontrivial, and the cone $T(\r_X)$ differs from the entire space. Notice that each motion $g\in G_{\r_X}$ takes $T(\r_X)$ into itself. Indeed, since $g(\r_X)=\r_X$, then the set of degenerate triangles in $X$ is mapped into itself by any permutation $g$ of points of the space $X$; this proves the invariance of $T(\r_X)$. Recall that for small $\e>0$ it holds $B_\e(\r_X)\cap T(\r_X)=B_\e(\r_X)\cap\cC_n$. Thus, for sufficiently small $\e>0$ the stabilizer $G_{\r_X}$ acts on the set $B_\e(\r_X)\cap T(\r_X)=B_\e(\r_X)\cap\cC_n$.

\begin{cor}\label{cor:LocIsomDegNonReg}
For each non-regular degenerate space $X\in\cM_{[n]}$, for all sufficiently small $\e>0$ the closed ball $B_\e(X)\ss\cM_{[n]}$ is isometric to the space $\bigl[B_\e(\r_X)\cap T(\r_X)\bigr]/G_{\r_X}$ obtained from the intersection of the ball $B_\e(\r_X)$ in $\R^N_\infty$ with the cone $T(\r_X)$ by factorisation over action of the stabilizer $G_{\r_X}$ of the point $\r_X$.
\end{cor}

\subsection{More on Generic Spaces}

Now, let us apply Corollary~\ref{cor:proj-is-loc-isom-triv-stab}.

\begin{cor}\label{cor:proj-metric-props}
For any $X\in\cM_{[n]}^g$ there exists $\e>0$ such that
\begin{enumerate}
\item For any $\r\in\Pi^{-1}(X)$ the ball $B_\e(\r)$ in $\R^N_\infty$ lies entirely in $\cC_n^g$, and the set $\Pi^{-1}\bigl(B_\e(X)\bigr)$ equals to disjoint union of the balls $\bigl\{B_\e(\r)\bigr\}_{\r\in\Pi^{-1}(X)}$.
\item The restriction $\pi_{\e,\r}\:B_\e(\r)\to B_\e(X)$ of the projection $\Pi$ is an isometry.
\item The restriction $g_{\e,\r}\:B_\e(\r)\to B_\e\bigl(g(\r)\bigr)$ of the mapping $g\in G$ is also an isometry.
\item The mappings $\pi_{\e,\r}$ and $\pi_{\e,g(\r)}$ are agreed with each other in the following sense\/\rom: $\pi_{\e,\r}=\pi_{\e,g(\r)}\c g_{\e,\r}$, thus each mapping $\pi^{-1}_{\e,g(\r)}\c\pi_{\e,\r}$ coincides with the restriction of the mapping $g\in G$ onto the ball $B_\e(\r)$.
\end{enumerate}
\end{cor}

\begin{dfn}
We call \emph{canonical} each neighbourhood $B_\e(X)$ from Corollary~\ref{cor:proj-metric-props}, together with all neighbourhoods $B_\e(\r)$.
\end{dfn}

\begin{prop}\label{prop:CngLInearConnected}
The subsets $\cC_n^g\ss\R^N$ are path-connected for all $n\ne3$\rom; moreover, each pair of points in $\cC_n^g$ can be connected by a polygonal line lying in $\cC_n^g$. For $n=3$ the subset $\cC_n^g\ss\R^3$ is not path-connected. The subsets $\cM_{[n]}^g\ss\cM_{[n]}$ are path-connected for all $n$.
\end{prop}

\begin{proof}
If $n=1$ or $n=2$, then $\cC_n^g=\cC_n$ and $M_{[n]}^g=M_{[n]}$, thus the path-connectivity follows from the above remarks.

Let $n=3$. Show that $\cC_n^g$ is not path-connected. Take, for instance, two points $\r_0=(3,4,5)$ and $\r_1=(4,3,5)\in \cC_3^g$, and suppose that there exists a continuous curve $\r_t=\big(\r_{12}(t),\r_{13}(t),\r_{23}(t)\big)$, $t\in[0,1]$, that lies in $\cC_3^g$ and connects these points. Then the continuous function $f(t)=\r_{12}(t)-\r_{13}(t)$ satisfies $f(0)<0$ and $f(1)>0$, therefore there exists $s\in(0,1)$ such that $\r_{12}(s)=\r_{13}(s)$. But then the stabilizer of the point $\r_s$ is nontrivial, thus $\r_s\not\in C_3^g$.

Now, show that $M_{[3]}^g$ is path-connected. To start with, notice that a triple of real numbers $a\le b\le c$ are the lengths of a triangle $X\in\cM_{[3]}^g$, iff $0<a<b<c<a+b$. Choose $X_0, X_1\in\cM_{[3]}^g$, and let $0<a_i<b_i<c_i<a_i+b_i$ be nonzero distances in $X_i$. Then for each $t\in[0,1]$ the triple $\{a_t=(1-t)a_0+t\,a_1,\,b_t=(1-t)b_0+t\,b_1,\,c_t=(1-t)c_0+t\,c_1\}$ also satisfies $0<a_t<b_t<c_t<a_t+b_t$ and, thus, it generates a metric space $X_t$ belonging to $\cM_{[3]}^g$. It is easy to see that $t\mapsto X_t$ is a continuous curve in $\cM_{[3]}^g$, therefore, $\cM_{[3]}^g$ is path-connected.

Consider the case $n\ge4$. Notice that the cone $\cC_n$ is convex, because it is the intersection of half-spaces corresponding to the positivity conditions of metric components, and to triangle inequalities. This implies that for any $\r_0,\r_1\in \cC_n$ the segment $\r_t=(1-t)\r_0+t\,\r_1$, $t\in[0,1]$, belongs to $\cC_n$. Further, if $\r_0,\r_1\in\cC_n$ are non-degenerate, then all $\r_t$ are non-degenerate as well. Thus, the set of all non-degenerate vectors $\r\in\cC_n$ is convex. Moreover, the set of all non-degenerate vectors $\r\in\cC_n$ is open and everywhere dense in $\cC_n$.

Now, let us investigate the structure of the set of all non-regular $\r\in\cC_n$. The condition of non-regularity of $\r\in\cC_n$ means that there exists a non-identical transformation $\s\in S_n$, such that $\s(\r)=\r$. Put $X=\Pi(\r)$ and let $\r=\r_X$ for some numeration $X=\{x_1,\ldots,x_n\}$ of points of the $X$, i.e., $\r_{ij}=|x_ix_j|$. Since the permutation $\s$ is not identical, then there exists $i\in\{1,\ldots,n\}$ such that $j=\s(i)\ne i$. Since $n\ge4$, then there exist at least two distinct $p,q\in\{1,\ldots,n\}$ different from $i$ such that $r=\s(p)\ne i$, $s=\s(q)\ne i$. This implies that $\{i,p\}\ne\{j,r\}$ and $\{i,q\}\ne\{j,s\}$, therefore, since $\s\bigl(\r_{ip}\bigr)=\r_{jr}$ and $\s\bigl(\r_{iq}\bigr)=\r_{js}$, the condition $\s(\r)=\r$ implies at least two non-identical conditions, namely, $\r_{ip}=\r_{jr}$ and $\r_{iq}=\r_{js}$.  Moreover, by assumption $i$ differs from $j$, $r$, and $s$, therefore all the four pairs $\{i,p\}$, $\{i,q\}$, $\{j,r\}$, and $\{j,s\}$ are pairwise distinct, and hence these two conditions are independent.  Therefore, the set of the vectors $\r\in\cC_n$ such that $\s(\r)=\r$, consists of subsets of a finite number of linear subspaces in $\R^N$ of codimension at least $2$. Those linear subspaces we call \emph{irregularity subspaces}.

Take two arbitrary $\r_1,\r_2\in\cC_n^g$, and for $\r_1$ and each irregularity subspace consider their linear hull. We get a collection of subspaces of nonzero codimensions. This implies that the union $W$ of those subspaces does not cover any open set in $\R^N$. Thus, since $\cC_n^g$ is open, there exists $\r'_2\in U_\e(\r_2)\ss\cC_n^g$, which does not belong to $W$. Therefore, the segment $[\r_1,\r'_2]$ does not intersect $W$, and, thus, the polygonal line $\r_1\r'_2\r_2$ does not intersect $W$ as well. This completes the proof that $\cC_n^g$ is path-connected and that each two its points can be connected by a polygonal line lying in $\cC_n^g$. Since the path-connectivity is preserved under continuous mappings, the set $\cM_{[n]}^g$ is path-connected also.
\end{proof}

\subsection{Coverings and Generic Spaces}
By $\Pi^g\:\cC_n^g\to\cM_{[n]}^g$ we denote the restriction of the mapping $\Pi\:\cC_n\to\cM_{[n]}$ onto $\cC_n^g$. Recall a definition of covering, see~\cite{FFG} for details.

Let $T$ and $B$ be path-connected topological spaces, $F$ be a discrete topological space, $n=\#F$. Then each continuous surjective mapping $\pi\:T\to B$ is called an \emph{$n$-sheeted covering with the total space $T$, the base $B$, and the fiber $F$}, if each point $b\in B$ has a neighborhood $U$ such that $\pi^{-1}(U)$ is homeomorphic to $U\x F$, and if $\v\:\pi^{-1}(U)\to U\x F$ is the corresponding homeomorphism, and $\pi_1\:U\x F\to U$ is the projection, $\pi_1\:(u,f)\mapsto u$, then $\pi=\pi_1\c\v$ (the corresponding diagram is commutative). If we omit the path-connectivity condition, then the mapping $\pi$ is called a \emph{covering in the broad sense}.

\begin{cor}\label{cor:loc-iso-covering}
The mapping $\Pi^g\:\cC_n^g\to\cM_{[n]}^g$ is an $n!$-sheeted locally isometric covering\/ \(in a broad sense for $n=3$, because $\cC_3$ is not path-connected\/\).
\end{cor}

We use Corollary~\ref{cor:loc-iso-covering} for constructing the lift of paths.

\begin{prop}[Lifting of paths~\cite{FFG}]\label{prop:th-cov-homot}
Let $\pi\:T\to B$ be an arbitrary covering in a broad sense, $\g\:[a,b]\to B$ be a continuous mapping\/ \(a path in $B$\), and $t\in T$ be an arbitrary point in $\pi^{-1}\bigl(\g(a)\bigr)$. Then there exists unique  continuous mapping $\G\:[a,b]\to T$, such that $\G(a)=t$ and $\g=\pi\c\G$.
\end{prop}

\begin{dfn}
The mappings $\G$ from Proposition~\ref{prop:th-cov-homot} is called the \emph{lift of $\g$}.
\end{dfn}

\section{Invariancy of $\cM_{[n]}^g$}
\markright{\thesection.~Invariancy of $\cM_{[n]}^g$}

In this Section we prove that the sets $\cM_{[n]}^g$ are invariant under any isometry of the space $\cM$. To do that, we use the technique elaborated above together with the invariancy of the Hausdorff measure under isometries. Recall the corresponding concepts and facts.

Let $X$ be an arbitrary set, and $2^X$ be the set of all subsets of $X$.

\begin{dfn}
An {\em outer measure\/} on the set $X$ is a mapping $\mu\:2^X\to[0,+\infty]$ such that
\begin{enumerate}
\item $\mu(\0)=0$;
\item for any at most countable family $\cC$ of subsets of $X$ and any $A\ss X$ such that $A\ss\cup_{B\in\cC}B$, it holds $\mu(A)\le\sum_{B\in\cC}\mu(B)$ (subadditivity).
\end{enumerate}
\end{dfn}

\begin{dfn}
A subset $A\ss X$ is called \emph{measurable with respect to $\mu$}, or simply {\em $\mu$-measurable\/} (\emph{in the sense of Carath\'eo\-dory\/}), if for any $Y\ss X$ it holds $\mu(Y)=\mu(Y\cap A)+\mu(Y\sm A)$.
\end{dfn}

\begin{dfn}
A family $\cS$ of subsets of $X$ is called a {\em $\s$-algebra on $X$}, if it contains $\0$, $X$, and it is closed under taking the complement and countable union operations.
\end{dfn}

It is well-known that for any outer measure $\mu$ on a set $X$ the set of all $\mu$-measurable subsets of $X$ is a $\s$-algebra. Also it is well-known that the intersection of any $\s$-algebras is a $\s$-algebra. The latter allows to define the smallest $\s$-algebra containing a given family of subsets of $X$.

If $X$ is a topological space, then the smallest $\s$-algebra containing the topology is called a \emph{Borel $\s$-algebra}, and its elements are called \emph{Borel sets}. An outer measure $\mu$ on a topological space is said to be \emph{Borel}, if all Borel sets are $\mu$-measurable. An outer measure $\mu$ on a topological space $X$ is said to be \emph{Borel regular}, if it is  Borel and for any set $A\ss X$ the value $\mu(A)$ is equal to the infimum of the values $\mu(B)$ over all Borel sets $B\sp A$.

Let $X$ be an arbitrary metric space. For our purposes it suffices to define the Hausdorff measure upto a multiplicative constant. For the standard definition of this measure see, for instance~\cite{BurBurIva}.

\begin{dfn}
For $\dl>0$ and $A\ss X$, a family $\{A_i\}_{i\in I}$ of subsets of $X$ is called a \emph{$\dl$-covering of the set $A$}, if $A\ss\cup_{i\in I}A_i$ and $\diam A_i<\dl$ for all $i\in I$ (if $A_i=\0$, then put $\diam A_i=0$).
\end{dfn}

\begin{dfn}\label{dfn:HausdorffMeasure}
For any $\dl>0$, $k>0$, and $A\ss X$ put
\begin{align}
&\label{align:Hk:1}H^k_\dl(A)=\inf\biggl\{\sum_{i=1}^\infty(\diam A_i)^k\,:\,\text{$\{A_i\}_{i=1}^\infty$ is a $\dl$-covering of $A$}\biggr\},\\
&\label{align:Hk:2}H^k(A)=\sup_{\dl>0}H^k_\dl(A).
\end{align}
\end{dfn}

The next results are well-known see, for example~\cite{BurBurIva}.

\begin{prop}\label{prop:first-prop-Hausdorff-mes}
For any $k>0$ and any positive integer $N$ the following conditions hold.
\begin{enumerate}
\item\label{prop:first-prop-Hausdorff-mes:1} For any metric space $X$ the functions $H^k$ are Borel regular outer measures on $X$.
\item\label{prop:first-prop-Hausdorff-mes:2} If $f\:X\to Y$ is an isometry of arbitrary metric spaces, then $H^k\bigl(f(A)\bigr)=H^k(A)$ for any subset $A\ss X$.
\item\label{prop:first-prop-Hausdorff-mes:3} In any $N$-dimensional normed space, the $H^N$-measure of a unit ball is nonzero and finite, thus, the $H^N$-measure of any bounded subset with nonempty interior is nonzero and finite.
\end{enumerate}
\end{prop}

\begin{prop}\label{prop:decreasing-Hk}
Suppose that a compact group $G$ acts continuously by isometries on a metric space $X$, and let $Y$ be a subset of $X$ that is invariant with respect to the group $G$ action. Suppose also that
\begin{enumerate}
\item\label{prop:decreasing-Hk:1} for some $k>0$ we have $H^k(Y)\in(0,\infty)$\rom;
\item\label{prop:decreasing-Hk:2} there exist $g\in G$ and $A\ss Y$ such that $H^k(A)>0$ and $A\cap g(A)=\0$.
\end{enumerate}
Then $H^k(Y/G)<H^k(Y)$.
\end{prop}

\begin{prop}\label{prop:GenPosIsoInvar}
Let $f\:\cM\to\cM$ be an arbitrary isometry, then $\cM_{[n]}^g=f(\cM_{[n]}^g)$.
\end{prop}

\begin{proof}
Choose an arbitrary $X\in\cM_{[n]}^g$ and let $Y=f(X)$. At first suppose that  $Y$ is a regular degenerate space, then, by Corollaries~\ref{cor:LocIsomGenPos} and~\ref{cor:LocIsomDegReg}, there exists $\e>0$ such that the ball $B_\e(X)\ss\cM_{[n]}$ is isometric to $B_\e(\r_X)\ss \R^N_\infty$, and the ball $B_\e(Y)$ is isometric to the intersection $B_\e(\r_Y)\cap T(\r_Y)$ of the ball $B_\e(\r_Y)\ss\R^N_\infty$ and the cone $T(\r_Y)$. Since the translations in $\R^N_\infty$ are isometries, then, by Items~(\ref{prop:first-prop-Hausdorff-mes:2}) and~(\ref{prop:first-prop-Hausdorff-mes:3}) of Proposition~\ref{prop:first-prop-Hausdorff-mes}, we have
$$
H^N\bigl(B_\e(\r_Y)\bigr)=H^N\bigl(B_\e(\r_X)\bigr)=H^N\bigl(B_\e(X)\bigr)=H^N\bigl(B_\e(Y)\bigr)=H^N\bigl(B_\e(\r_Y)\cap T_{\r_Y}\bigr)>0.
$$
By Proposition~\ref{prop:BeCAPTrX-nonempty-inter} and Item~(\ref{prop:first-prop-Hausdorff-mes:3}) of Proposition~\ref{prop:first-prop-Hausdorff-mes}, is holds
$$
H^N\Bigl(B_\e(\r_Y)\sm\bigl(B_\e(\r_Y)\cap T(\r_Y)\bigr)\Bigr)>0,
$$
therefore, since the outer measure $H^N$ is a Borel one, we get $H^N\bigl(B_\e(\r_Y)\cap T(\r_Y)\bigr)<H^N\bigl(B_\e(\r_Y)\bigr)$, a contradiction. Thus, the balls $B_\e(X)$ and $B_\e(Y)$ are not isometric, so $Y$ cannot be a regular degenerate space.

Next, let $Y$ be a non-regular non-degenerate space. Then, by Corollary~\ref{cor:LocIsomNonDegNonReg}, the ball $B_\e(Y)$ is isometric to $B_\e(\r_Y)/G_{\r_Y}$, where $G_{\r_Y}$ is the stabilizer of the point $\r_Y$, which is a nontrivial group, because the space $Y$ is non-regular. Let $g\in G_{\r_Y}$ be an element different from the unity. Since the generic spaces are everywhere dense in $\cM_{[n]}$, there exists $\r_Z\in U_\e(\r_Y)$ corresponding to a generic space $Z\in\cM_{[n]}$. Since the stabilizer of the point $\r_Z$ is trivial, Item~(\ref{prop:compact-group-action-stab:2}) of Proposition~\ref{prop:compact-group-action-stab} implies that there exists $\dl>0$ such that $U_\dl(\r_Z)\ss U_\e(\r_Y)$ and $g\bigl(U_\dl(\r_Z)\bigr)\cap U_\dl(\r_Z)=\0$. However, $U_\dl(\r_Z)$ is an open ball in $\R^N_\infty$,  therefore, by Item~(\ref{prop:first-prop-Hausdorff-mes:3}) of Proposition~\ref{prop:first-prop-Hausdorff-mes}, we have $0<H^N\bigl(U_\dl(\r_Z)\bigr)<\infty$. Further, by Item~(\ref{prop:compact-group-action-stab:1}) of Proposition~\ref{prop:compact-group-action-stab}, it holds $g\bigl(U_\dl(\r_Z)\bigr)\ss U_\e(\r_Y)$, thus, by Item~(\ref{prop:decreasing-Hk:2}) of Proposition~\ref{prop:decreasing-Hk} we conclude that $H^N\bigl(B_\e(\r_Y)/G_{\r_Y}\bigr)<H^N\bigl(B_\e(\r_Y)\bigr)$ and, so,
$$
H^N\bigl(B_\e(Y)\bigr)=H^N\bigl(B_\e(\r_Y)/G_{\r_Y}\bigr)<H^N\bigl(B_\e(\r_Y)\bigr)=H^N\bigl(B_\e(\r_X)\bigr)= H^N\bigl(B_\e(X)\bigr).
$$
Thus, $Y$ cannot be a non-regular non-degenerate space.

The case of a non-regular degenerate space $Y$ can be proceeded by a combination of the above arguments.
\end{proof}

\section{Local Affinity Property}
\markright{\thesection.~Local Affinity Property}

In 1968~\cite{John} F.~John obtained a generalization of the Mazur--Ulam Theorem~\cite{MazurUlam} on affinity property of isometries of normed vector spaces.

\begin{prop}[\cite{John}, Theorem IV, p. 94]\label{prop:John}
Let $U\ss X$ be a connected open subset of a real complete normed space $X$, and $h\:U\to W$ be an isometry that maps $U$ onto an open subset $W$ of a real complete normed space $Y$. Then $h$ is the restriction of an affine isometry $H\:X\to Y$.
\end{prop}

Proposition~\ref{prop:John} implies that all the isometries of the space $\R^d_\infty$ are affine. Describe these isometries in more derails.

\begin{prop}\label{prop:max-isom}
Let $h\:\R^d_\infty\to\R^d_\infty$ be an affine isometry. Then $h(x)=(S\cdot P)x+b$, where $b\in\R^d$ is a translation vector, $P$ is a permutation matrix of the vectors from the standard basis, and $S$ is a diagonal matrix with $\pm1$ on its diagonal.
\end{prop}

\begin{proof}
Any affine mapping is a composition of a linear mapping $x\mapsto Ax$ with a translation by some vector. Since the distance in a normed space is invariant under any translation, it suffices to describe all linear isometries $h(x)=Ax$. Every such mapping takes the unit ball with center at the origin onto itself. Notice that the ball in $\R^d_\infty$ is the cube with vertices at points with coordinates $\pm1$. The hyperfaces (the facets) of this cube are given by the equations $x_i=\pm1$, and $h$ maps them into each other. This implies that the faces of the cube (of any dimension) are transferred by $h$ into the faces of the same dimension.

The center of a hyperface $x_i=\pm1$ is the vector $\pm e_i$, where $e_i$ is a vector in the standard basis of the arithmetic space $\R^d$. Notice that this center is equal to the sum of radius vectors of the corresponding hyperface vertices, up to the factor $2^{d-1}$. Thus the mapping $h$ takes each vector $e_i$ into a vector $\pm e_j$, i.e., $h$ is the composition of a basic vectors permutation with their signs changes.
\end{proof}

Let $f\:\cM\to\cM$ be an arbitrary isometry, $X\in\cM_{[n]}^g$ and $Y=f(X)$. By Proposition~\ref{prop:GenPosIsoInvar}, we have $Y\in\cM_{[n]}^g$. Choose $\e>0$ in such a way that the balls $B_\e(X)$ and $B_\e(Y)$ in $\cM_{[n]}$ are canonical neighbourhoods. Then for any $\r_X\in\Pi^{-1}(X)$ and $\r_Y\in\Pi^{-1}(Y)$ we have $B_\e(\r_X)\ss\cC_n^g$,  $B_\e(\r_Y)\ss\cC_n^g$, and the restrictions $\pi_{\e,\r_X}$ and $\pi_{\e,\r_Y}$ of the mapping $\Pi^g$ onto these neighbourhoods are isometries with $B_\e(X)$ and $B_\e(Y)$, respectively. However, in this case the mapping
$$
h_{\e,\r_X,\r_Y}=\pi^{-1}_{\e,\r_Y}\c f\c\pi_{\e,\r_X}\:U_\e(\r_X)\to U_\e(\r_Y)
$$
is an isometry as well. By Proposition~\ref{prop:John}, the mapping $h$ is affine. Thus, we get the following result.

\begin{cor}\label{cor:affine-form}
Under the above notations, if $\e>0$ is such that $B_\e(X)$ and $B_\e(Y)$ are canonical neighbourhoods, then the mapping
$$
h_{\e,\r_X,\r_Y}=\pi^{-1}_{\e,\r_Y}\c f\c\pi_{\e,\r_X}\:U_\e(\r_X)\to U_\e(\r_Y)
$$
has the form $h_{\e,\r_X,\r_Y}(\r)=(S\cdot P)\r+b$, where $b\in\R^N$ is a translation vector, $P$ is a permutation matrix of the standard basic vectors, and $S$ is a diagonal matrix with $\pm1$ on its diagonal.
\end{cor}

The next Lemma will be used in what follows.

\begin{lem}\label{lem:affine-maps-intersect-coincide}
If two affine mappings $x\mapsto A_i\,x+b_i$, $i=1,2$, defined on intersecting open subsets of the space $\R^d$ are coincide in the intersection, then $A_1=A_2$ and $b_1=b_2$.
\end{lem}

\begin{constr}\label{constr:path-in-cC}
Let $X,X'\in\cM_{[n]}^g$ and the corresponding $\r_X,\r_{X'}\in\cC_n^g$ are such that the segment $L=[\r_X,\r_{X'}]$ belongs to $\cC_n^g$. Let us consider the segment $L$ as a continuous curve, and denote by $\g$ the image of the curve $L$ under the mapping $\Pi^g$. Then $\g$ is a curve in $\cM_{[n]}^g$ joining $X$ and $X'$. Let $\g'$ be the image of the curve $\g$ under the isometry $f$, then $\g'$ joins $Y:=f(X)$ and $Y':=f(X')$. Choose an arbitrary $\r_Y\in\cC_n^g$. By Corollary~\ref{cor:loc-iso-covering}, the mapping $\Pi^g\:\cC_n^g\to\cM_{[n]}^g$ is a covering in a broad sense, therefore, by Proposition~\ref{prop:th-cov-homot}, there exists a unique continuous curve $L'$ in $\cC_n^g$ starting at $\r_Y$ and such that its $\Pi^g$-image is the curve $\g'$. Since the second endpoint of the curve $L'$ is projected to $Y'$, this endpoint coincides with $\r_{Y'}$ for some numeration of points of the space $Y'$.

Now, choose $\e>0$ such that all the balls $U_\e(X)$, $U_\e(X')$, $U_\e(Y)$, and $U_\e(Y')$ are canonical neighbourhoods simultaneously. Then, under the notations of Corollary~\ref{cor:proj-metric-props}, the isometries $\pi_{\e,\r_X}$, $\pi_{\e,\r_{X'}}$, $\pi_{\e,\r_Y}$, $\pi_{\e,\r_{Y'}}$ generate two other isometries $h_{\e,\r_X,\r_Y}=\pi^{-1}_{\e,\r_Y}\c f\c\pi_{\e,\r_X}$ and $h_{\e,\r_{X'},\r_{Y'}}=\pi^{-1}_{\e,\r_{Y'}}\c f\c\pi_{\e,\r_{X'}}$. By Proposition~\ref{prop:John}, the mappings $h_{\e,\r_X,\r_Y}$ and $h_{\e,\r_{X'},\r_{Y'}}$ are the restrictions of affine isometries $H\:\R^N_\infty\to\R^N_\infty$ and $H'\:\R^N_\infty\to\R^N_\infty$, respectively.
\end{constr}

\begin{lem}\label{lem:affine-mappings-coincides-segments}
Under the above notations, the affine isometries $H$ and $H'$ coincide.
\end{lem}

\begin{proof}
Let the segment $L$, together with the curves $\g$, $\g'$, and $L'$, are parameterized by a parameter $t\in[a,b]$, $L(a)=\r_X$ and $L(b)=\r_{X'}$.

For each $t\in[a,b]$ choose $\e_t>0$ such that $B_{\e_t}\bigl(L(t)\bigr)$ and $B_{\e_t}\bigl(L'(t)\bigr)$ are canonical neighbour\-hoods. The family of balls $\Bigl\{U_{\e_t}\bigl(L(t)\bigr)\Bigr\}$ is an open covering of the segment $[\r_X,\,\r_{X'}]$. Let $\{U_i\}_{i=1}^m$ be a finite subcovering that exists, because the segment is compact. Without loss of generality, suppose that the family $\{U_i\}_{i=1}^m$ is minimal in the sense that no one $U_i$ is contained in another $U_j$; besides, assume that the centers $\r_i$ of the balls $U_i$ are ordered along the segment $[\r_X,\r_{X'}]$. These two conditions imply that the consecutive $U_i$ intersect each other, in particular, the distance between each $\r_i$ and $\r_{i+1}$ is less than the sum of radii $\e_i$ and $\e_{i+1}$ of the balls $U_i$ and $U_{i+1}$, respectively. Since the balls $U_i$ are open, then each intersection $U_i\cap U_{i+1}$ is open as well. Further, since $|\r_i\r_{i+1}|_\infty<\e_i+\e_{i+1}$, then there exists $\r'_i\in(\r_i,\r_{i+1})$ such that $\r'_i\in U_i\cap U_{i+1}$; besides, since the set $U_i\cap U_{i+1}$ is open, one can choose an open ball $U'_i$ with center $\r'_i$ and radius $\e'_i$ in such a way that $U'_i\ss\ U_i\cap U_{i+1}$. As a result, we have constructed a new covering $\{U_1,U'_1,U_2,U'_2,\ldots\}$ of the segment $[\r_X,\,\r_{X'}]$. By $\{V_i\}_{i=1}^{2m-1}$ we denote the consecutive elements of this new covering. Introduce new notations: let $\r_i=L(t_i)$ be the center of the ball $V_i$, and $\e_i$ be the radius of this ball. Thus, $V_i=U_{\e_i}(\r_i)$.

Further, put $\nu_i=L'(t_i)$ and consider the family of open balls $\bigl\{W_i:=U_{\e_i}(\nu_i)\bigr\}_{i=1}^{2m-1}$, then, by definition, each of these balls lies in $\cC_n^g$, and the restrictions $\pi_{\e_i,\nu_i}$ of the mapping $\Pi^g$ onto these balls are isometric. Besides, the restriction $\pi_{\e_i,\r_i}$ of the mapping $\Pi^g$ onto $V_i$ is isometric as well. Put $h_i=h_{\e_i,\r_i,\nu_i}=\pi^{-1}_{\e_i,\nu_i}\c f\c\pi_{\e_i,\r_i}$, then $h_i\:V_i\to W_i$ is an isometry for each $i$.

Since, by the construction $V_{2k}\ss V_{2k-1}$, the fact that $h_{2k-1}$ is an isometry implies that $|\r_{2k}\r_{2k-1}|_\infty=|\nu_{2k}\nu_{2k-1}|_\infty$, therefore, $W_{2k}\ss W_{2k-1}$ and $\pi_{\e_{2k},\nu_{2k}}=\pi_{\e_{2k-1},\nu_{2k-1}}|_{W_{2k}}$, because the both mappings are the restrictions of $\Pi^g$. Thus, $h_{2k}=h_{2k-1}|_{V_{2k}}$. Similarly, one can show that $h_{2k}=h_{2k+1}|_{V_{2k}}$.

By Proposition~\ref{prop:John}, for every $i$ there exists an affine mappings $H_i\:\R^N_\infty\to\R^N_\infty$ such that $h_i=H_i|_{V_i}$. As  we have shown above, the consecutive $h_i$ coincide on open sets that are the intersections of the domains of the corresponding mappings, thus, by Lemma~\ref{lem:affine-maps-intersect-coincide}, all these $H_i$ coincide, in particular, $H_1=H_{2m-1}$. By the same Lemma, $H=H_1$ and $H'=H_{2m-1}$.
\end{proof}

\begin{cor}\label{cor:affine-mappings-coincides-polyline}
If in Construction~$\ref{constr:path-in-cC}$ one changes the segment $L$ by a polygonal line, then Lemma~$\ref{lem:affine-mappings-coincides-segments}$ remains true.
\end{cor}

\begin{prop}\label{prop:affine-form}
Under the notations of Corollary~$\ref{cor:affine-form}$, the matrix $S$ is unit, and $b=0$.
\end{prop}

\begin{proof}
Under the notations of Construction~\ref{constr:path-in-cC}, let us choose an arbitrary $0<\dl<1$ and take the space $\dl\,X$ as $X'$. By Corollary~\ref{cor:affine-mappings-coincides-polyline}, the mappings
$$
h_{\e,\r_X,\r_Y}\:U_\e(\r_X)\to U_\e(\r_Y)\quad \text{and} \quad h_{\e,\r_{X'},\r_{Y'}}\:U_\e(\r_{X'})\to U_\e(\r_{Y'})
$$
are the restrictions of the same affine isometry $H$ which does not depend on the choice of $\dl$. By Proposition~\ref{prop:max-isom}, we have $H(\r)=(S\cdot P)\r+b$, where $b\in\R^N_\infty$ is a translation vector, $P$ is a permutation matrix of the standard basic vectors, and $S$ is a diagonal matrix with $\pm1$ on its diagonal.

Notice that $\|\dl\,\r_X\|_\infty\to0$ as $\dl\to0$, therefore, by Corollary~\ref{cor:fD1}, we have $\bigl\|H(\dl\,\r_X)\bigr\|_\infty\to0$ as $\dl\to0$. However, if $b\ne0$, then for $\dl$ such that
$$
\|\dl\,\r_X\|_\infty<\frac12\|b\|_\infty\quad \text{and} \quad \bigl\|(S\cdot P)(\dl\,\r_X)\bigl\|_\infty<\frac12\|b\|_\infty,
$$
we get
$$
\bigl\|H(\dl\,\r_X)\bigl\|_\infty=\bigl\|(S\cdot P)(\dl\,\r_X)+b\bigl\|_\infty\ge
-\bigl\|(S\cdot P)(\dl\,\r_X)\bigl\|_\infty+\|b\|_\infty>\frac12\|b\|_\infty,
$$
a contradiction with that $\bigl\|H(\dl\,\r_X)\bigr\|_\infty\to0$ as $\dl\to0$.

Thus, we have shown that $b=0$. It remains to note that all the components of the vectors $\r_X$ and $H(\r_X)$ are positive, therefore, $S$ cannot contain negative elements.
\end{proof}

\section{Some Necessary Facts on Permutation Groups}
\markright{\thesection.~Some Necessary Facts on Permutation Groups}

We have shown above that each mapping $h_{\e,\r_X,\r_Y}$ is the restriction of an affine mapping $x\mapsto Px$ of the space $\R^N$ onto itself, where $P$ is a permutation matrix of the basis vectors, i.e., in fact, an element from the permutation group $S_N$. To complete the proof, we show that $P\in G$, i.e., $P$ is generated by a permutation (i.e., renumeration) of points of metric spaces and, thus, $f$ is an identical mapping locally. To do that, we need a number of facts on permutation groups.

Put $V=\{1,\ldots,n\}$. Let $E$ be the set of the basis vectors $e_{ij}=e_{ji}$ of the space $\R^N$, $i\ne j$. Identify $e_{ij}$ with the corresponding two-element subset $\{i,j\}\ss V$. Then $K_n=(V,E)$ is a complete graph with $n$ vertices and $N$ edges, and hence, the actions of $G$ and $S_N$ can be considered as actions on the set of edges $E$ of the graph $K_n$; note that the action of the group $G$ is generated by permutations on the vertices set $V$. Notice that for $n=3$ all six permutations of the edges are generated by the permutations of vertices, i.e., in this case $G=S_N$.
	
\begin{lem}\label{lem:neibor_preserv}
Let $n\ge 5$. A permutation $\a\in S_N$ belongs to the subgroup $G$, iff $\a$ takes adjacent edges to adjacent ones.
\end{lem}

\begin{proof}
It is easy to see that each permutation $\a\in G$ takes adjacent edges to adjacent ones. Now, let us prove the converse statement.

Suppose that $\a$ takes all pairs of adjacent edges of the graph $K_n$ to adjacent ones. Consider all the edges incident to some fixed vertex $v\in V$ (the number of such edges is $n-1$, in particular, it is not less than $4$). Then, by assumption, their images are pairwise adjacent. Let us show that the edges--images also have a common vertex.

Consider images of any three different edges from the chosen ones. Their images form a connected three-edge subgraph $H$ of $K_n$. Each such subgraph is either a cycle, or a star, or a simple path. The latter case is impossible, because the first and the last edges of the path are not adjacent.

Consider now the image of a fourth edge. It has to be adjacent with all three edges of the subgraph $H$. Therefore, $H$ cannot be a three-edge cycle and, thus, $H$ is a star, and the image of the fourth edge is incident to the common vertex of the star.

Arguing in a similar way, we come to conclusion that the images of all the edges incident with the vertex $v$ are incident to some common vertex. Thus, it is defined a mapping $\s$ from the set $V$ onto itself taking each vertex  $v\in V$ to the unique common vertex of the $\a$-images of all the edges incident to $v$. This mapping is injective: indeed, if $v$ and $w$ are mapped to the same vertex, then their image is common for $2n-2$ edges that is impossible. Besides, $\s$ induces a mapping on the edges of the graph $K_n$ which coincides with $\a$, therefore, $\a\in G$.
\end{proof}

\begin{rk}
For $n=4$ the condition of Lemma~\ref{lem:neibor_preserv} is not sufficient. For instance, the next permutation $\a$ takes triples of edges having common vertex to triples of edges which form cycles:
$$
\a=\Bigl(\begin{array}{cccccc}
\{1,2\} & \{1,3\} & \{1,4\} & \{2,3\} & \{2,4\} & \{3,4\} \\
\{2,3\} & \{3,4\} & \{2,4\} & \{1,3\} & \{1,2\} & \{1,4\}
\end{array}
\Bigr).
$$
Evidently, $\a$ takes adjacent edges to adjacent ones, but it is not generated by a permutation of vertices.
\end{rk}

\begin{ass}\label{ass:normal}
For $n\ge 8$ the normalizer of the subgroup $G$ in $S_N$ coincides with the group $G$.
\end{ass}

\begin{proof}
By $F$ we denote the set of all pairs of different edges of the graph $K_n$. Then $F=F_0\sqcup F_1$, where $F_0$ consists of the pairs of non-adjacent edges, and $F_1$ consists of the pairs of adjacent edges (i.e., of edges having a common vertex).

\begin{lem}\label{lem:pairs_quantities}
Under the above notations,
$$
\#F_0=\frac{n(n-1)(n-2)(n-3)}{8},\qquad \#F_1=\frac{n(n-1)(n-2)}{2}.
$$
In particular, for $n\ge 8$ the number of pairs of non-adjacent edges is greater than the number of pairs of adjacent ones, i.e., $\#F_0>\#F_1$.
\end{lem}

\begin{proof}
Indeed, consider the graph $E(K_n)$, whose vertices are the edges of the graph $K_n$, and two its vertices are adjacent, iff the corresponding edges of $K_n$ are adjacent. Then $E(K_n)=(E,F_1)$. Each edge $\{i,j\}$ of $K_n$ is adjacent in $E(K_n)$ with $n-2$ edges by the vertex $i$, and with $n-2$ edges by the vertex $j$, thus the degree of each vertex of the graph $E(K_n)$ equals $2n-4$. Since the number of vertices of the graph $E(K_n)$ equals $N=n(n-1)/2$, then we get:
$$
\#F_1=\frac{1}{2}\cdot\frac{n(n-1)}{2}\cdot(2n-4)=\frac{n(n-1)(n-2)}{2}.
$$
To calculate the number of the pairs of non-adjacent edges, we have to subtract $\#F_1$ from the number of all pairs:
\begin{multline*}
\#F_0=\#F-\#F_1=\frac{\frac{n(n-1)}{2}\Big(\frac{n(n-1)}{2}-1\Big)}{2}-\frac{n(n-1)(n-2)}{2}=\\
=\frac{n(n-1)}{2}\Big(\frac{n(n-1)}{4}-\frac12-(n-2)\Big)=\frac{n(n-1)}{2}\cdot\frac{n^2-5n+6}{4}=\frac{n(n-1)(n-2)(n-3)}{8}.
\end{multline*}
In particular,
$$
\#F_0=\frac{(n-3)}{4}\cdot\#F_1,
$$
therefore, since $(n-3)/4>1$  for $n>7$, we have $\#F_0>\#F_1$ for $n\ge8$.
\end{proof}

Each permutation $P\in S_N$ acts not only on the edges of the graph $K_n$, but also on the set $F$ by means of the same rule. Lemma~\ref{lem:pairs_quantities} implies the following result.

\begin{lem}\label{lem:pair_non_pair}
For $n\ge 8$ each permutation $P\in S_N$ takes some pair of non-adjacent edges of the graph $K_n$ to a pair of its non-adjacent edges.
\end{lem}

\begin{proof}
Indeed, the number of elements in $F_1$ is smaller then the one for $F_0$, so we do not have enough elements in $F_1$ to map all elements of $F_0$ onto them.
\end{proof}

Return to the proof of Assertion~\ref{ass:normal}. Let the permutation $P\in S_N$ satisfies $P^{-1}gP\in G$ for some $g\in G$. If $P\not\in G$, then by Lemma~\ref{lem:neibor_preserv}, the permutation $P$ takes some two adjacent edges $\{i,j\}$ and $\{i,k\}$ of the graph $K_n$ to a pair of non-adjacent edges $\{a,b\}$ and $\{c,d\}$. Besides, by Lemma~\ref{lem:pair_non_pair}, there exist two non-adjacent edges $\{i_1,j_1\}$ and $\{i_2,j_2\}$ which are mapped by $P$ to non-adjacent edges $\{a',b'\}$ and $\{c',d'\}$. Since $a$, $b$, $c$, and $d$, as well as $a'$, $b'$, $c'$, and $d'$ are pairwise distinct, then there exists a permutation $g\in S_n$ such that
$$
g(a)=a',\quad g(b)=b', \quad g(c)=c', \quad g(d)=d'.
$$
Then
$$
\{i,j\}\stackrel{P}{\longmapsto}\{a,b\}\stackrel{g}{\longmapsto}\{a',b'\}\stackrel{P^{-1}}{\longmapsto}\{i_1,j_1\}
$$
and
$$
\{i,k\}\stackrel{P}{\longmapsto}\{c,d\}\stackrel{g}{\longmapsto}\{c',d'\}\stackrel{P^{-1}}{\longmapsto}\{i_2,j_2\},
$$
i.e., the composition $P^{-1}g\,P$ takes some adjacent edges to non-adjacent ones, and, thus, is does not belong to $G$. This contradiction completes the proof of the assertion.
\end{proof}

\section{Completion of Main Theorem Proof}
\markright{\thesection.~Completion of Main Theorem Proof}
Let $f\:\cM\to\cM$ be an isometry, $X\in\cM^g_{[n]}$, and $f(X)=Y$. Choose $\e>0$ such that $B_\e(X)$ and $B_\e(Y)$ are canonical neighbourhoods. Fix some $\r_X\in\Pi^{-1}(X)$ and $\r_Y\in\Pi^{-1}(Y)$, then the mapping $h_{\e,\r_X,\r_Y}\:U_\e(\r_X)\to U_\e(\r_Y)$ from Corollary~\ref{cor:affine-form} is the restriction of a linear mapping $\R^N\to\R^N$ with permutation matrix $P\in S_N$ (this linear mapping we denote by the same letter $P$).

\begin{lem}\label{lem:P_isom}
For $n\ge4$ the permutation $P\in S_N$ belongs to the normalizer of the subgroup $G$, i.e., $P^{-1}gP\in G$ for every $g\in G$.
\end{lem}

\begin{proof}
It is easy to see that the subset of $\cM_{[n]}^g$ consisting of all spaces such that all their nonzero distances are pairwise distinct, is everywhere dense in $\cM_{[n]}^g$. Besides, if $Z$ is such a space, then for any numeration of the points from $Z$, all the components of the vector $\r_Z$ are pairwise distinct, therefore, each $Q\in S_N$ is uniquely defined by the $Q$-image of such point $\r_Z$.

Chose $X\in\cM_{[n]}^g$ in such a way that all nonzero distances in $X$ are pairwise distinct. By Proposition~\ref{prop:CngLInearConnected} and Corollary~\ref{cor:affine-mappings-coincides-polyline}, for any $\r\in\Pi^{-1}(X)$ there exists $\r'\in\Pi^{-1}(Y)$ such that the mapping $h_{\e,\r,\r'}\:U_\e(\r)\to U_\e(\r')$ is the restriction of a linear mapping with the same matrix $P$. Therefore, $P\bigl(\Pi^{-1}(X)\bigr)\ss\Pi^{-1}(Y)$. Since the matrix $P$ is non-degenerate, then for any distinct $\r_1,\r_2\in\Pi^{-1}(X)$ we have $P(\r_1)\ne P(\r_2)$. At last, since $\#\Pi^{-1}(X)=\#\Pi^{-1}(Y)$, then $P$ maps $\Pi^{-1}(X)$ bijectively onto $\Pi^{-1}(Y)$.

Thus, $\r'_X:=P^{-1}gP(\r_X)\in\Pi^{-1}(X)$, and this means that there exists $g'\in G$ such that $\r'_X=g'(\r_X)$. However, as we mentioned above, the mapping $P^{-1}gP\in S_N$ is uniquely defined by the image of $\r_X$. Thus,  $P^{-1}gP=g'\in G$.
\end{proof}

Now, Lemma~\ref{lem:P_isom} and Assertion~\ref{ass:normal}imply that for $n\ge8$ the permutation $P$ is contained in $G$, therefore the vectors $\rho_X$ and $\rho_{f(X)}$ differ by a renumeration of vertices, i.e., $X=f(X)$. Thus, we have shown that the isometry $f$ is fixed on an everywhere dense subset of the space $\cM_{[n]}^g$ and, thus, on the entire $\cM_{[n]}^g$. It remains to notice that the union $\cup_{n\ge8}\cM_{[n]}^g$ is everywhere dense in $\cM$. Main Theorem is proved.

\end{document}